\documentclass[a4paper,oneside]{article}

\usepackage{amsmath}%
\usepackage{amssymb}%
\usepackage[latin1]{inputenc}
\usepackage[numbers]{natbib}

\usepackage{geometry}
\newtheorem{theorem}{Theorem}[section]

\newtheorem{definition}[theorem]{Definition}

\newtheorem{lemma}[theorem]{Lemma}

\newtheorem{proposition}[theorem]{Proposition}
\newtheorem{remark}[theorem]{Remark}

\newenvironment{proof}[1][Proof]{\textbf{#1.} }{\ \rule{0.5em}{0.5em}}

\newcommand{\refeqn}[1]{(\ref{#1})}
\newcommand{\virgolette}[1]{``#1''}
\newcommand{\cinf}[0]{C^{\infty}}

\newcommand{\spann}[0]{\operatorname{span}}

\newcommand{\jinf}[0]{J^{\infty}}

\begin{document}

\title{{\bf Finite dimensional solutions to SPDEs and the geometry of infinite jet bundles}}
\author{ Francesco C. De Vecchi\thanks{Dipartimento di Matematica, Universit\`a degli Studi di Milano, via Saldini 50, Milano, Italy, \emph{email: francesco.devecchi@unimi.it, francesco.devecchi.fdv@gmail.com}}}
\date{}

\maketitle

\begin{abstract}
Finite dimensional solutions to a class of stochastic partial differential equations are obtained extending  the differential constraints method
for deterministic PDE to the stochastic framework. A geometrical reformulation of the stochastic problem using the concept of infinite jet
bundles is provided and a practical algorithm for explicitly computing these finite dimensional solutions is developed. This method, covering
the majority of the current literature, is applied to a set of new SPDEs admitting finite dimensional solutions taken from Heath-Jarrow-Morton
framework, stochastic hydrodynamics and filtering theory.
\end{abstract}

\bigskip


\section{Introduction}

Lie symmetry analysis for ordinary and partial differential equations
(ODEs and PDEs respectively)  and its generalizations  are nowadays a classical research topic in applied mathematics, since they provide powerful and flexible tools
both for studying the qualitative behaviour of differential equations and for obtaining some explicit expression of their solutions
(see, e.g., \cite{Gaeta,Olver1993,Stephani1989}).\\
The exploitation of these techniques for studying stochastic differential equations (SDEs) is  more recent but is arousing a growing
interest (see, e.g., \cite{Albeverio,DeLara1995,DMU1,Gaeta2017} for Brownian-motion-driven SDEs and
\cite{Albeverio2017,Glover1990,Cami2009,Liao2009} for SDEs driven by more general semimartingales). This topic finds  compelling applications in
the detection and study of stochastic systems having closed formulas or a great analytical tractability (see, e.g.,
\cite{Craddock2007,Craddock2012} where Lie symmetry techniques are applied to the research of the transition probability density of diffusion processes,
\cite{DMU2,Kozlov2010,Cami2009} where they are exploited for reducing a SDEs to quadratures, \cite{DeVecchi2017,Malham2008} where they are
used for improving the standard numerical schemes for integrating SDEs and
\cite{Albeverio,Lescot,Zambrini2010} where they are applied in order to find conservation laws or martingales related to the considered process).\\
Although the notion of infinitesimal symmetry finds many applications in the theory of finite dimensional SDEs,
to the best of our knowledge and with the notable exception of the works of Cohen De Lara (see \cite{DeLara2,Delara1998}),
there is no direct extension of Lie symmetry techniques for PDEs in the stochastic partial differential equations (SPDEs) setting.\\

\noindent In this paper we propose a first application of the geometric methods developed in the study of infinitesimal symmetries of PDEs to SPDEs. In
particular we generalize the following well known property of symmetric PDEs: the knowledge of a symmetry algebra for a given PDE can be used to
reduce the PDE to a system of ODEs by looking for  solutions which are invariant with respect to the symmetry algebra (see \cite{Olver1993}).
In the deterministic framework this kind of reduction has
been generalized with the introduction of the concept of differential constraints (see
\cite{Galaktionov,Meleshko,Olver1}) and in the following we use these deterministic ideas for explicitly constructing (when possible) special solutions to an SPDE reducing it to a finite dimensional SDE.\\

\noindent One of the main differences between the stochastic and the deterministic setting is that, while the solutions to an evolution PDE are in some respect one dimensional (and so it is always possible to reduce a PDE to a finite dimensional ODE), the solutions to a generic SPDE cannot be described using only a finite dimensional stochastic system since they usually spread in all the infinite dimensional space where the SPDE is set. For this reason, if there exist some solutions to an SPDE which can be described using a finite dimensional SDE, these solutions are called \emph{finite dimensional solutions} to the considered SPDE. The problem of finding SPDEs admitting finite dimensional solutions and the explicit computation of these type of solutions has a long tradition in stochastic analysis and finds
interesting applications in several mathematical modelling. \\
For example this problem arises  in stochastic filtering, in particular in the research of finite dimensional filters, i.e. stochastic filters described by a finite number of parameters such as the Kalman filter or the Benes filter (see \cite{Bain}). This problem is equivalent to studying the finite dimensional solutions to a special linear SPDE called Zakai equation (see the classical works \cite{Albeverio1990,Benes1981,Brockett1981,Hazewinkel1989,Yau1994} and the papers of Cohen De Lara \cite{DeLara1bis,DeLara2}).\\
A second application of the problem of finding finite dimensional solutions to SPDEs is the Heath-Jarrow-Morton (HJM) model used in mathematical finance for describing the interest rate (see \cite{Heath1992}). In this setting it is important to understand when a given finite dimensional family of curves can consistently describe the interest rate (this is the consistency problem studied in \cite{Bjork1999,Filipovic2001}). This problem is equivalent to establishing when a special class of SPDEs admits finite dimensional solutions (see in particular the works \cite{Filipovic1,Filipovic3,Filipovic2,Tappe2016} of Filipovic, Tappe and Teichmann). \\
A third application of our results is to the study of stochastic soliton equations. In addition to the pioneering work of Wadati on the stochastic KdV equation preserving soliton solutions (see \cite{Wadati,Xie}), we have also been inspired by the recent growing  interest in the study
of variational stochastic systems of hydrodynamic  type (see e.g. \cite{Arnaudon2012,Holm2015,Cruzeiro2016}). In particular we recall \cite{Holm2016}, where Holm and Tyranowski find many families of finite dimensional soliton type solutions to a physically important
stochastic perturbation of Camassa-Holm equation (see also \cite{Crisan2017}).\\

\noindent In this paper we consider  SPDEs of the form
\begin{equation}\label{equation_introduction_SPDE1}
dU^i_t(x)=\sum_{\alpha=1}^r F^i_{\alpha}(x,U_t(x),\partial^{\sigma}(U_t(x))) \circ dS^{\alpha}_t,
\end{equation}
where $x \in M=\mathbb{R}^m$, $U_t(x)$ is a semimartingale taking values in  $N=\mathbb{R}^n$, $F^i_{\alpha}(x,u,u_{\sigma})$ are smooth
functions of the independent coordinates $x^i$, the dependent coordinates $u^i$ and their derivatives $u^i_{\sigma}$ (here $\sigma \in
\mathbb{N}_0^m$ is a multi-index denoting the numbers of derivatives with respect to the coordinates $x^i$), $S^1,...,S^r$ are $r$ continuous
semimartingales and $\circ$ denotes the Stratonovich integration.\\
 In this framework  we are interested in determining finite dimensional solutions to SPDE
\refeqn{equation_introduction_SPDE1} and we tackle this problem using the geometry of the
infinite jet bundle $\jinf(M,N)$ of the functions defined on $M \subset \mathbb{R}^m$ and taking values in $N \subset \mathbb{R}^n$. Jet
bundles have been introduced by Charles Ehresmann and provide a very useful tool for a modern approach to Lie symmetry analysis allowing a natural geometric
interpretation of deterministic differential equations. The infinite jet bundle $\jinf(M,N)$ is an infinite dimensional manifold modelled on
$\mathbb{R}^{\infty}$ whose main advantage, with respect to the usual infinite dimensional (Banach or
Fr\'echet) functional spaces such as $L^2(M,N)$ or $\cinf(M,N)$, consists in the existence of a simple coordinate system which can be exploited in explicit computations. \\
In this setting we prove that looking for finite dimensional solutions to SPDE
\refeqn{equation_introduction_SPDE1} is equivalent to
establishing  under which conditions the solution process $U_t(x)$ to
SPDE \refeqn{equation_introduction_SPDE1} can be written in the form
$$U_t(x)=K(x,B^1_t,...,B^k_t),$$
where $K:M \times \mathbb{R}^k \rightarrow N$ is a smooth function of all its variables and $B_t=(B^1_t,...,B^k_t) \in \mathbb{R}^k$ is a
stochastic process satisfying a suitable finite dimensional SDE.\\
Working with the infinite jet bundle formalism, an SPDE becomes an infinite dimensional (ordinary) SDE in $\jinf(M,N)$ and the function $K(x,b)$ can be naturally associated with a finite dimensional
 submanifold $\mathcal{K}$ of $\jinf(M,N)$. Therefore we can approach the problem of finding finite dimensional solutions to equation \refeqn{equation_introduction_SPDE1} in a  completely
geometrical way. Indeed, with the functions $F_{\alpha}=(F^1_{\alpha},...,F^{n}_{\alpha})$ defining the SPDE
\refeqn{equation_introduction_SPDE1} it is possible to associate a set of vector fields $V_{F_{\alpha}}$ defined in $\jinf(M,N)$, and the
probabilistic problem of finding finite dimensional solutions to SPDE \refeqn{equation_introduction_SPDE1} is equivalent to find a
submanifold $\mathcal{K}$ such that the vector
fields $V_{F_{\alpha}}$ are tangent to $\mathcal{K}$. \\

\noindent It is important to note that our geometrical reinterpretation of the problem needs a suitable definition of solutions to SPDE
\refeqn{equation_introduction_SPDE1}. In order to give such a definition we use the notion of semimartingales smoothly depending  on some spatial
parameters proposed by Kunita in \cite{Kunita1990}. With this probabilistic tool we can give a rigorous sense to the intuitive definition of
solution to SPDE \refeqn{equation_introduction_SPDE1} based on the idea of taking a process $U_t(x)$ depending both on $t$ and $x$ and smooth
with respect to $x$ and then verifying equation \refeqn{equation_introduction_SPDE1} by replacing $U_t(x)$ in \refeqn{equation_introduction_SPDE1}
for any fixed $x \in M$. We compare this notion of solution with the usual ones based on the martingale calculus in Hilbert spaces of
\cite{DaPrato1992}, proving their equivalence under simple common hypotheses on the Hilbert space and on the process $U_t(x)$.\\

\noindent Once the probabilistic problem has been transformed into a geometric one, we can tackle the latter using natural tools developed in the geometric
theory of deterministic PDE. First of all we propose a necessary condition on the coefficients $F^i_{\alpha}$ (see
Proposition \ref{proposition_necessary}) for the existence of finite dimensional solutions to SPDE \refeqn{equation_introduction_SPDE1}.
Furthermore we provide a sufficient condition (proved in our paper \cite{DMevolution}) for the existence of finite dimensional solutions to SPDEs. This sufficient condition requests that
the vector fields $V_{F_{\alpha}}$ form a finite dimensional Lie algebra and they admit characteristic flow (the notion of characteristic flow
of a vector field in $\jinf(M,N)$ is a generalization of the usual definition of characteristics of a first order scalar PDE in the
$\jinf(M,N)$ setting). Under these hypotheses, in Theorem \ref{theorem_main1} and Theorem \ref{theorem_main2}, for any smooth initial
condition for equation \refeqn{equation_introduction_SPDE1}, we are able to construct a finite dimensional manifold $\mathcal{K}$ which
guarantees the existence of finite dimensional solutions to the considered SPDE.\\

\noindent We exploit the explicit construction method given in the proof of  Theorem \ref{theorem_main1} and Theorem \ref{theorem_main2} in order to develop  a practical algorithm (see
Section \ref{section_general_algorithm}), which  is applied to three selected examples taken from three different classical fields (HJM
theory, hydrodynamic and filtering theory) where finite dimensional solutions to SPDEs have their applications. The first example
is a model for HJM theory with proportional volatility, considered by Morton in his thesis \cite{Morton1989}, for which we give for the first
time, to the best of our knowledge, an explicit solution formula. The second example is a stochastic perturbation of the Hunter-Saxton equation
which is a simplification of the stochastic Camassa-Holm equation considered in \cite{Holm2016}. The third example is inspired by
filtering theory, and is an extension of the well known formulas of Fourier transform of affine processes (see \cite{affine}).\\

\noindent The methods proposed in this paper have been  deeply inspired by the previous works on finite dimensional solutions to SPDEs. In
particular our setting can be seen as a non-trivial generalization of the results proposed in \cite{DeLara1bis,DeLara2} by Cohen De Lara for
studying Zakai
equation to the case of general non-linear SPDEs of the form \refeqn{equation_introduction_SPDE1}.  Moreover, the
  relation between the  Lie algebra generated by the operators $F_i$ and the existence of finite dimensional filters  can be found
  in the classical  literature on the subject (see \cite{Benes1981,Brockett1981,Yau1994}). Indeed the necessary conditions obtained in Proposition
  \ref{proposition_lie_algebra} are, in the case of Zakai equation, equivalent to the conditions obtained for the existence of finite dimensional filters. \\
Furthermore the works of Filipovic, Tappe and Teichmann about finite dimensional solutions to HJM equation triggered a part of the paper.
In particular Theorem \ref{theorem_main1} and Theorem \ref{theorem_main2} are reformulations of \cite{Filipovic2}, where the use of the
convenient  setting  of global analysis (\cite{Michor1997}) is replaced by the infinite jet bundle geometry and  the characteristics of
Section \ref{section_characteristics}.\\
On the other hand our work introduces some novelties. First of all we propose an unified point of view on the subject which provides, for some
respects,  a generalization of the current literature. Indeed the form of equation \refeqn{equation_introduction_SPDE1} is completely general
and includes as special cases both the Zakai equation considered by Cohen De Lara and the semilinear SPDEs considered by Filipovic, Tappe and
Teichmann. Furthermore, Theorem \ref{theorem_main1} and Theorem \ref{theorem_main2} allow us to construct all the smooth solutions considered
by the previous methods. Nevertheless, our perspective should be considered as complementary and not as alternative to the previous ones.
Indeed we consider only smooth solutions to SPDE \refeqn{equation_introduction_SPDE1}: Theorem \ref{theorem_main1} is proved only in smooth
setting,  although, restricting the generality of equation \refeqn{equation_introduction_SPDE1}, it could be extended to the non-smooth framework. Moreover,
using Theorem \ref{theorem_main1}, we are able to construct \emph{one} solution to equation \refeqn{equation_introduction_SPDE1} among the many
possible smooth solutions with the same initial data. In fact, if we do not restrict the class of the possible solutions to a suitable space of
functions, we have not a uniqueness result for equation \refeqn{equation_introduction_SPDE1}. For this reason, once we construct the solution with
our method we should, a posteriori, prove that the solution belongs to a suitable space of functions where a uniqueness result for
SPDE \refeqn{equation_introduction_SPDE1} holds. This feature is a consequence of the generality of our methods: indeed, if we are interested in constructing
solutions belonging to a given class of functions, we should use different (more analytic) methods such as those proposed in the
previous literature.\\
A second novelty of our perspective is that we provide an algorithm for the explicitly computation of  finite dimensional solutions to SPDEs
which covers all the relevant cases considered in the current literature. Furthermore we propose new examples of interesting SPDEs, among which
all the  SPDEs considered in Section \ref{section_examplesbis}, as well as  HJM model considered in Section \ref{subsection_HJM}, whose
explicit solution was not known.\\

\noindent We conclude the introduction with a notice to the reader. This work aim at providing both a theoretical framework for handle the problem of finite dimensional solutions to SPDEs and a practical and general method for finding these solutions. For this reason the work can be divided in two parts: the first part (Sections \ref{prel}, \ref{section_finite} and \ref{section_characteristics}) deals with the theoretical foundation while the second part (Sections \ref{section_general_algorithm} and \ref{section_examplesbis}) exhibits the practical algorithm. The reader  only interested in  computational aspects, after reading Section \ref{subsection_informal}, can skip the first part and look directly the second one.\\

\noindent The paper is organized as follows. Section \ref{prel} introduces the preliminaries on the geometry of infinite jet bundles necessary in our theory. The relationship between the infinite jet bundles and the problem of finding finite dimensional solutions to SPDEs is discussed in Section \ref{section_finite}. In Section \ref{section_characteristics} we propose some general theorems for the construction of finite dimensional solutions to SPDEs and, in Section \ref{section_general_algorithm}, we use these results to develop an explicit algorithm. Finally, in Section \ref{section_examplesbis}, we apply our results  to three selected examples.

\section{Preliminaries}\label{prel}

In this section we collect  some basic facts about (infinite) jet bundles in order to provide the necessary geometric tools for our aims.
 Einstein summation convention  over repeated indices in used throughout the paper.
\\

\subsection{An informal introduction to the geometry of $\jinf(M,N)$}\label{subsection_informal}

We start with an informal introduction to the geometry of $\jinf(M,N)$, where $M,N$ are two open subsets of $\mathbb{R}^m,\mathbb{R}^n$
respectively.
The main advantage of the  infinite jet bundle setting, with respect to the analytic Fr\'echet
spaces approach, relies on the  computational aspects which turn out to be definitely simpler.\\

Let $C^k(M,N)$ be the infinite dimensional Fr\'echet space  of $k$ times differentiable functions defined on $M$ and taking values in $N$. We can associate with
$C^k(M,N)$ the  finite dimensional manifold $J^k(M,N)$
identifying  $f,g:M \rightarrow N$ whenever $g(x_0)=f(x_0)$ and
$\partial^{\sigma}(f)(x_0)=\partial^{\sigma}(g)(x_0)$, where $\sigma \in \mathbb{N}_0^n$ is a multi-index with $|\sigma|=\sum_r \sigma_r
\leq k$. The space $J^k(M,N)$ is called $k$-jets bundle of functions from $M$ into $N$ and can be endowed with a natural coordinate system.
If $x^i$ is the standard  coordinate system on $M$ (the space
of independent variables) and $u^j$ is the standard  coordinate system on $N$ (the space of dependent variables), a coordinate
system on $J^k(M,N)$ is given  by $x^i,u^j$ and all the variables $u^i_{\sigma}$, where $|\sigma|\leq k$, which formally represent the  derivative of the functions $u^j(x)$.
The smooth manifold $J^k(M,N)$ is a smooth vector bundle on $M$ with  projection $\pi_{k,-1}:J^k(M,N)
\rightarrow M$ given by
$$\pi_{k,-1}(x^i,u^j,u_{\sigma}^j)=x^i.$$
With any function $f \in C^k(M,N)$ we can associate a continuous section of the bundle
$(J^k(M,N),M,\pi_{k,-1})$ in the following way
$$ f \longmapsto D^k(f)(x)=(x, u^j=f^j(x),u^j_{\sigma}=\partial_{\sigma}(f^j)(x)).$$
Moreover, for any $k,h \in \mathbb{N}$ with $h < k$, there is a natural projection $\pi_{k,h}: J^k(M,N) \to J^h(M,N)$ given by
$$\pi_{k,h}(x^i,u^j,(u^j_{\sigma})|_{|\sigma|\leq k})=(x^i,u^j,(u^j_{\sigma})|_{|\sigma|\leq h}).$$
This allows us to consider the space $\jinf(M,N)$  defined  as the inverse limit of the sequence of projections
$$M \stackrel{\pi_0}{\leftarrow} M \times N=J^0(M,N) \stackrel{\pi_{1,0}}{\leftarrow} J^1(M,N) \stackrel{\pi_{2,1}}{\leftarrow} ...
\stackrel{\pi_{k,k-1}}{\leftarrow} J^k(M,N) \stackrel{\pi_{k+1,k}}{\leftarrow} ... $$
Analogously to $J^k(M,N)$, also $\jinf(M,N)$ has
a natural coordinate system given by $x^i,u^j$ and $u^j_{\sigma}$, with no bound on $|\sigma|$.\\
Since $\jinf(M,N)$ is not a finite dimensional manifold, but a Fr\'echet manifold modelled on $\mathbb{R}^{\infty}$
(see, e.g., \cite{Hamilton1982} for an introduction to the concept), working with  spaces of
smooth functions defined on $\jinf(M,N)$ is quite difficult. On the other hand,   the explicit coordinate system on $\jinf(M,N)$ suggests the possibility of restricting
 to a suitable space of smooth functions on $\jinf(M,N)$ which permits explicit
calculations.
In fact, if we consider the space
$$ \mathfrak{F}=\bigcup_{k}\mathfrak{F}_k,$$
where $\mathfrak{F}_k$ is the set of smooth functions defined on $J^k(M,N)$, i.e. $F \in \mathfrak{F}_k$ if $F$ is of the form
$F(x^i,u^j,u^j_{\sigma})$ with $|\sigma|\leq k$,  $\mathfrak{F}$ is the set of functions  depending only on a finite subset of coordinates $x^i,u^j,u^j_{\sigma}$.
Given any vector field $V \in T\jinf(M,N)$  of the form
$$V=\phi^i \partial_{x^i}+\psi^j \partial_{u^j}+\psi^j_{\sigma} \partial_{u^j_{\sigma}},$$
where $\phi^i, \psi^j, \psi^j_{\sigma}$ are smooth functions on $\jinf(M,N)$,  if $\phi^i, \psi^j, \psi^j_{\sigma} \in \mathfrak{F}$,  we have that  $V(\mathfrak{F}) \subset \mathfrak{F}$.
In the following we only consider vector fields $V$ whit $\phi^i, \psi^j, \psi^j_{\sigma} \in \mathfrak{F}$.\\
Therefore, given two vector fields $V_1,V_2$, we can define a Lie bracket  given by
$$[V_1,V_2]=(V_1(\phi^i_{2})-V_2(\phi^i_1))\partial_{x^i}+(V_1(\psi^j_2)-V_2(\psi^j_1))\partial_{u^j}+(V_1(\psi^j_{\sigma,2})-V_2(\psi^j_{\sigma,1}))
\partial_{u^j_{\sigma}}.$$
We recall that in $\jinf(M,N)$ one can naturally define the formally integrable Cartan distribution $\mathcal{C}=\spann\{D_1,...,D_m \}$ generated by  the vector fields
$$D_i=\partial_{x^i}+\sum_{k,\sigma} u^k_{\sigma+1_i} \partial_{u^k_{\sigma}}$$
satisfying $[D_i,D_j]=0$.
Another important class of vector fields in $\jinf(M,N)$ is given by the vector fields $V^e$  commuting with all $D_i$. It is possible to prove that $V^e$
commutes with all $D_i$ if and only if $V^e$ is of the form
$$V^e=F^j\partial_{u^j}+D^{\sigma}(F^j) \partial_{u^j_{\sigma}},$$
where $F^j \in \mathfrak{F}$. We say that $V^e$ is an evolution vector field generated by the function $F=(F^1,...,F^n) \in \mathfrak{F}^n$ and we write $V^e=V_F$. The Lie
brackets between two evolution vector fields is a new evolution vector field. This means that, for any $F,G \in \mathfrak{F}^n$, there exists a unique
function $H \in \mathfrak{F}^n$ such that $[V_F,V_G]=V_H$. Denoting by $H=[F,G]$, it is simple to prove that the brackets $[\cdot,\cdot]$  make $\mathfrak{F}^n$
an infinite dimensional Lie algebra.\\

Using   the natural projection $\pi_k:\jinf(M,N) \rightarrow J^k(M,N)$ of $\jinf(M,N)$ on $J^k(M,N)$, it is possible to
define a useful notion of smooth submanifold of $\jinf(M,N)$. A subset $\mathcal{E}$  of $\jinf(M,N)$ is a submanifold of $\jinf(M,N)$ if, for
any $p \in \mathcal{E}$, there exists a neighborhood $U_p$ of $p$ such that $\pi_h(\mathcal{E} \cap U_p)$ is a
submanifold of $J^h(M,N)$ for $h>H_p$.\\
If, for any $p \in \mathcal{E}$, all the submanifolds $\pi_h(\mathcal{E} \cap U_p)$ with $h>H_p$   have the
 same finite dimension $L$, we say that $\mathcal{E}$ is an $L$-dimensional submanifold of $\jinf(M,N)$.
In particular, given an $L$-dimensional  manifold $B$ and a smooth immersion $K:B \rightarrow \jinf(M,N)$, for any point $y \in B$ there exists
a neighborhood $V$ of $p$ such that $K(V)$ is a finite dimensional submanifold of $\jinf(M,N)$. A vector field $V \in T\jinf(M,N)$ is tangent to
the submanifold $\mathcal{E}$ if, for any $h \in \mathfrak{F}$ such that $h|_{\mathcal{E}}=0$, we have  $V(h)|_{\mathcal{E}}=0$. In this case we write $Y \in T\mathcal{E}$.

\begin{definition}\label{definition canonical_submanifold}
A submanifold $\mathcal{E}$ of $\jinf(M,N)$ is  a \emph{canonical submanifold} in and only if  $\mathcal{C} \subset T\mathcal{E}$.
Any canonical submanifold $\mathcal{E}$ can  be locally described as the set of zeros of a finite number of smooth independent
functions $f_1,...,f_L$ and of all their differential consequences $D^{\sigma}(f_i)$.\\
\end{definition}

  A  finite dimensional smooth canonical submanifold $\mathcal{K}$   is called integral manifold of the Cartan distribution. In order to construct an integral manifold of the Cartan
distribution  we recall that  $\jinf(M,N)$ is a smooth bundle on $M$ with  projection $\pi:\jinf(M,N) \rightarrow M$ such that
$\pi(x^i,u^j,u^j_{\sigma})=x^i $. Analogously to the case of  finite jets spaces, we can define the operator $D^{\infty}(f)$  associating with any $f \in
\cinf(M,N)$ a smooth section $D^{\infty}(f)$ of the bundle $(\jinf(M,N),M,\pi)$ in the natural way.
Given  $f \in \cinf(M,N)$, we define
$$\mathcal{K}^f=\bigcup_{x \in M}(x,D^{\infty}(f)(x)).$$
We have that $\mathcal{K}^f$ is an $n$ dimensional submanifold of $\jinf(M,N)$ and $D_i \in T\mathcal{K}^f$. In fact, if $F \in \mathfrak{F}$, the vector fields $D_i$ satisfy
$$D_i[F](x,,f(x),\partial^{\sigma}(f)(x))=\partial_{x^i}\left[F(x,f(x),\partial^{\sigma}(f)(x))\right],$$
for any $f \in \cinf(M)$.
On the other hand, if $\mathcal{K}$ is an integral manifold of $\mathcal{C}$,  there exist a unique function $f^{\mathcal{K}} \in \cinf(M,N)$ such that
$\mathcal{K}^{f^{\mathcal{K}}}=\mathcal{K}$. In this way we can identify   any integral manifold of $\mathcal{C}$ with  a  smooth
function in $\cinf(M,N)$ or, equivalently, we can describe any smooth function as
 an integral manifold of $\mathcal{C}$ in $\jinf(M,N)$. \\

\begin{remark}\label{remark_evolution}
The previous considerations and the definitions of $D_i$ and  $D^{\infty}$ provide a natural interpretation for evolution vector fields.
In particular, if the function $f \in \cinf(M \times \mathbb{R}, N)$ solves an evolution equation of the form
$$\partial_t(f)(x,t)=F(x,f(x),\partial^{\sigma}(f)(x)),$$
it is easy  to prove that,  for any $G \in \mathfrak{F}$, we have
$$\partial^n_t\left[G(x,f(x),\partial^{\sigma}(f)(x))\right]=V_F^n[G](x,f(x),\partial^{\sigma}(f)(x)).$$
These properties will play an important role in the representation of SPDEs as ordinary SDEs on the infinite dimensional manifold $\jinf(M,N)$.
\end{remark}

In the following, in order to make the previous discussion more explicit, we rewrite the expressions of the principal objects introduced above in the particular case of $\jinf(\mathbb{R},\mathbb{R})$.
In the space $\jinf(\mathbb{R},\mathbb{R})$ we consider the coordinate system given by $x \in \mathbb{R}$ (the coordinate  on $M$), by $u \in \mathbb{R}$ (the coordinate  in $N$) and by all the (formal) derivatives of $u$ with respect to $x$ which are $u_{(1)},u_{(2)},u_{(3)},...$. Sometimes, in order to simplify the notation and clarify the meaning of the coordinate system $x,u,u_{(1)},...$ we write  $u_x=u_{(1)}$, $u_{xx}=u_{(2)}$,..., . \\
If $F \in  \mathfrak{F}$, then $F$ is a smooth function depending only on $x$, $u$ and the derivative $u_{(n)}$ for $n<k$, with $k$ an integer great enough. The vector field $D_1=D_x$ has the form
$$D_x=\partial_x+u_x\partial_{u}+u_{xx}\partial_{u_x}+...+u_{(n+1)}\partial_{u_{(n+1)}}+...$$
and represents the formal derivative with respect to $x$ in $\jinf(M,N)$, which means that, if $F(x,u,u_x,...) \in \mathfrak{F}$ and $f \in \cinf(\mathbb{R})$, then
$$D_x(F)(x,f(x),f'(x),...)=\partial_x(F(x,f(x),f'(x),...)).$$
In this case, the evolution vector field $V_F$ has the form
$$V_F=F\partial_u+D_x(F)\partial_{u_x}+...+D^n_x(F)\partial_{u_{(n)}}+...$$
In particular, if for example  $F=xu_x$, we have
$$V_F=xu_x\partial_u+(xu_{xx}+u_x)\partial_{u_x}+...+(xu_{(n+1)}+nu_{(n)})\partial_{u_{(n)}}+...$$
In this setting  there is a simple way to see finite dimensional canonical submanifolds of $\jinf(\mathbb{R},\mathbb{R})$ as ordinary differential equations of arbitrary order for the dependent variable $u$. Consider for example the subamanifold $\mathcal{K}^n$ in $J^n(\mathbb{R},\mathbb{R})$ defined as the set of zeros of the equation
\begin{equation}\label{equation_consequence1}
u_{(n)}-h(x,u,u_x,...,u_{(n-1)})=0,
\end{equation}
where $h \in \mathfrak{F}_{n-1}$. If we want  $\mathcal{K}^n$ to be the projection on $J^n(\mathbb{R},\mathbb{R})$ of some canonical submanifold $\mathcal{K}$ of $\jinf(\mathbb{R},\mathbb{R})$ we need that $D_x \in T\mathcal{K}$ and so
\begin{equation}\label{equation_consequence2}
\begin{array}{rcl}
0&=&D_x(u_{(n)}-h(x,...))=u_{(n+1)}-D_x(h)(x,...)\\
0&=&D^2_x(u_{(n)}-h(x,...))=u_{(n+2)}-D^2_x(h)(x,...)\\
&...&
\end{array}
\end{equation}
on $\mathcal{K}$. Equations \refeqn{equation_consequence2} are called differential consequences of equation \refeqn{equation_consequence1} and a finite dimensional canonical submanifold $\mathcal{K}$ is defined by equation \refeqn{equation_consequence1} and its differential consequences \refeqn{equation_consequence2}. It is possible to prove that the generic (with respect to a suitable topology) canonical submanifold of $\jinf(\mathbb{R},\mathbb{R})$ is of the form described above. In particular, in $\jinf(\mathbb{R},\mathbb{R})$, every canonical submanifold is finite dimensional (we remark that this is no more  true when $M$ is of dimension greater than one).

\subsection{Finite  dimensional  canonical submanifolds of $\jinf(M,N)$ and reduction functions}

In this section, generalizing the identification between integral manifolds of the contact distribution in $\jinf(M,N)$ and smooth functions,  we  prove that  any  $(m+r)$ dimensional canonical submanifold in $\jinf(M,N)$  can be identified  with a smooth function defined on $M \subset \mathbb{R}^m$ taking values in $N \subset \mathbb{R}^n$ and depending on $r$ parameters. \\
In fact, given  an $r$ dimensional smooth manifold $B$ and a smooth function
$$K:M \times B \rightarrow N,$$
which we call a \emph{finite dimensional function}, we can consider the function
$$\bold{K}:M \times B \rightarrow \jinf(M,N) $$
defined by
$$\bold{K}(x,b)=(x,D^{\infty}(K)(x,b))$$
and the subset
$$\mathcal{K}^K=\bigcup_{x \in M, b \in B} \bold{K}(x,b),$$
where the $D^{\infty}$ operator acts only on the $x^i$ variables of $K$.

\begin{theorem}\label{theorem_finite}
If $\mathcal{K}^K$ is a finite dimensional submanifold of $\jinf(M,N)$, then $\mathcal{K}^K$ is a finite dimensional canonical submanifold.
Conversely, if $\mathcal{K}$ is a finite dimensional canonical submanifold of $\jinf(M,N)$ then, suitably restricting  $M$ and
$\mathcal{K}$, there exists a finite dimensional function $K$ such that $\mathcal{K}^K=\mathcal{K}$.
\end{theorem}
\begin{proof}
The fact that, for any smooth finite dimensional function $K$, if $\mathcal{K}^K$ is a finite dimensional submanifold of $\jinf(M,N)$, then $\mathcal{K}^K$ is a finite dimensional canonical manifold follows from the fact that, for any fixed $b$, $D_i \in \bigcup_{x \in M} \bold{K}(x,b)$ since $\bold{K}(x,b)=(x,D^{\infty}(K)(x,b))$.\\
Conversely, let $\mathcal{K}$ be a finite dimensional canonical submanifold of $\jinf(M,N)$. By definition of submanifold of $\jinf(M,N)$, possibly restricting $M$, we can describe $\mathcal{K}$ as the set of zeros of some functions of the form
$$u^i_{\sigma}-f^i_{\sigma}(x^1,...,x^m,y^1,...,y^r),$$
where $y^1,...,y^r \in \mathfrak{F}_k$ and $f^i_{\sigma}:\mathbb{R}^{m+r} \rightarrow \mathbb{R}$ are smooth functions. Thanks to the previous property we can work in the finite dimensional manifold $J^k(M,N)$, rather than in the infinite dimensional $\jinf(M,N)$. If we choose an adapted coordinate system $x^1,...,x^n,y^1,...,y^r$  in $\mathcal{K}$,  the vector fields $D_i$ restricted to $\mathcal{K}$ will be of the form
$$D_i=\partial_{x^i}+\psi_i^k(y^1,...,y^r)\partial_{y^k},$$
for some functions $\psi_i^k$ (since $[D_i,D_j]=0$ the functions $\psi^k$ do not depend on $x^1,...,x^n$). Fixing $x_0 \in M$, there is only one solution to the following system of overdetermined PDEs
\begin{eqnarray*}
\partial_{x^i}(Y^k)(x,y^1_0,...,y^r_0)&=&\psi_i^k(x,y^1_0,...,y^r_0)\\
Y^k(x_0,y^1_0,...,y^r_0)&=&y^k_0,
\end{eqnarray*}
for $(y^1_0,...,y^r_0)$  in a suitable open subset of $\mathbb{R}^r$. Hence, if we restrict $M$ to a suitable neighborhood of $x_0$, any integral submanifold of $\mathcal{C}|_{\mathcal{K}}$ is of the form
$$\bigcup_{x \in M} (x,Y^1(x,y^1_0,...,y^r_0),...,Y^r(x,y^1_0,...,y^r_0)),$$
for some $y^1_0,...,y^r_0$. Since $x^i,y^j$ form a coordinate system on $\mathcal{K}$, the coordinate $u^i \in \mathfrak{F}$ restricted to $\mathcal{K}$ is of the form
$$u^j=\mathfrak{U}^j(x^1,...,x^n,y^1,...,y^r).$$
This means that the finite dimensional canonical manifold $\mathcal{K}$ is the canonical manifold generated by the function $K \in \cinf(M \times B,N)$ defined by
$$K^j(x,y^1_0,...,y^r_0)=\mathfrak{U}^j(x,Y^1(x,y^1_0,...,y^r_0,...,Y^r(x,y^1_0,...,y^r_0)).$$
${}\hfill$\end{proof}\\

\noindent The proof of Theorem \ref{theorem_finite} provides a constructive method to obtain the finite dimensional function associated with a finite dimensional canonical submanifold $\mathcal{K}$. This method is very simple in the case of $\jinf(\mathbb{R},\mathbb{R})$, and in the following we  give  the idea of the construction in an explicit case. Given $\lambda \in \mathbb{R}$, let $\mathcal{K}$  be defined by
$$u_{xx}-\lambda u=0,$$
and all its differential consequences. This means that $\mathcal{K}$ is defined by the equations
\begin{eqnarray*}
u_{(2n)}-\lambda^n u&=&0\\
u_{(2n+1)}-\lambda^n u_x&=&0.
\end{eqnarray*}
If we choose on $\mathcal{K}$ the coordinate system  $(x,u,u_x)$, the vector field $D_x$ restricted to $\mathcal{K}$ is given by
$$D_x=\partial_x+u_x\partial_u+\lambda u \partial_{u_x}.$$
In order to construct the function $K$ generating $\mathcal{K}$ we need to solve the differential equations
\begin{eqnarray*}
\partial_x(U(x,u_0,u_{x,0}))&=&U_x(x,u_0,u_{x,0})\\
\partial_x(U_x(x,u_0,u_{x,0}))&=& \lambda U(x,u_0,u_{x,0})\\
U(0,u_0,u_{x,0})&=&u_0\\
U_x(0,u_0,u_{x,0})&=&u_{x,0}.
\end{eqnarray*}
If, for example,  $\lambda > 0$, the solution to the previous system is
\begin{eqnarray*}
U(x,u_0,u_{x,0})&=&\frac{\sqrt{\lambda}u_0+u_{x,0}}{2\sqrt{\lambda}}e^{\sqrt{\lambda}x}+\frac{\sqrt{\lambda}u_0-u_{x,0}}{2\sqrt{\lambda}}e^{-\sqrt{\lambda}x}\\
U_x(x,u_0,u_{x,0})&=&\frac{\sqrt{\lambda}u_0+u_{x,0}}{2}e^{\sqrt{\lambda}x}+\frac{-\sqrt{\lambda}u_0+u_{x,0}}{2}e^{-\sqrt{\lambda}x}.
\end{eqnarray*}
Since in this case $\mathfrak{U}=u$, the finite dimensional function $K$ generating $\mathcal{K}$ is exactly
$$K(x,u_0,u_{x,0})=U(x,u_0,u_{x,0})=\frac{\sqrt{\lambda}u_0+u_{x,0}}{2\sqrt{\lambda}}e^{\sqrt{\lambda}x}+\frac{\sqrt{\lambda}u_0-u_{x,0}}{2\sqrt{\lambda}}e^{-\sqrt{\lambda}x}.$$
We remark  that, in the case $M \subset \mathbb{R}$, constructing  $K$  is equivalent to finding the solution to the ODE $u_{(n)}-h(x,...)=0$ defining, together with all its differential consequences, the manifold  $\mathcal{K}$ and, conversely, the finite dimensional canonical submanifold $\mathcal{K}$ associated with $K$ is the unique ODE for which $K$ is the fundamental solution.

\section{Finite dimensional solutions to SPDEs and finite dimensional canonical manifolds}\label{section_finite}

\subsection{SPDEs and the geometry of $\jinf(M,N)$}

Given a probability space $(\Omega,\mathcal{F},\mathbb{P})$ with a filtration of subsigmalgebras $\mathcal{F}_t \subset \mathcal{F}$,  in the following we consider only  (local) martingales which are (local) martingales with respect to the filtration $\mathcal{F}_t$.  In this setting
the definition of semimartingale with a spatial parameter proposed in
\cite{Kunita1990} can be modified as follows.

\begin{definition}\label{definition_semimartingale}
Let $(t,x,\omega) \mapsto U_{t}(x)(\omega) \in N$ be a random variable. We say that $U_t$ is a semimartingale dependent on the parameter $x \in
M$ of regularity $h$ (in short $U_t(x)$ is a $C^h$ semimartingale) if, for any $t \in \mathbb{R}_+, \omega \in \Omega$, the function $U_t(\cdot)(\omega) \in C^h(M,N)$ and, for any $x \in M$
and multi-index $|\sigma| \leq h$, the process $\partial^{\sigma}(U_t)(x)$ is an $N$ valued semimartingale. If $U_t$ is a $C^h$ semimartingale for any $h >0$, we say that $U_t$ is a $\cinf$ semimartingale.
\end{definition}

\begin{remark}
If $U_t$ is a $C^h$ semimartingale, then $\partial^{\sigma}(U_t)$ is a $C^{h-|\sigma|}$ semimartingale for any multi-index $|\sigma| \leq h$.
\end{remark}

\begin{definition}\label{definition_SPDE}
Given $F_1,...,F_r \in \mathfrak{F}^n_h$ and $r$ real semimartingales $S^1,...,S^r$, we  say that the $C^h$ semimartingale $U_t=(U^i_t,...,U^n_t)$ is a solution to the SPDE
\begin{equation}\label{equation_definitionSPDE1}
dU_t=F_{\alpha}(U_t) \circ dS^{\alpha}_t,
\end{equation}
or simply to the SPDE associated with $(F_1,...,F_r)$ and $(S_1,...,S_r)$ if and only if, for any $x \in N$,
\begin{equation}\label{equation_definitionSPDE2}
U^i_t(x)-U^i_0(x)=\int_0^t{F_{\alpha}(x,U_s(x),...,\partial^{\sigma}(U_s)(x)) \circ dS^{\alpha}_s}.
\end{equation}
\end{definition}

\begin{remark}\label{remark_SPDE1}
We can extend Definition \ref{definition_SPDE} to more general SPDEs. For example consider a functional $\Psi_{\alpha}:\cinf(M,N) \rightarrow
\mathbb{R}$. If we suppose that $\Psi(f)$ depends only on the values of the function $f$ in some compact subset $\mathfrak{K} \subset M$
and that $\Psi$ is smooth with respect to the norm of $C^{l}(K,N)$ (for some $l \geq 0$) it is easy to prove, using the It\^o formula for Hilbert space
valued semimartingale, that $\Phi_{\alpha}(U_t)$ is a real semimartingale. In this setting we can modify equation
\refeqn{equation_definitionSPDE2} in the following way
\begin{equation}\label{equation_nonlocal}
U^i_t(x)-U^i_0(x)=\int_0^t{\Psi_{\alpha}(U_s)F_{\alpha}(x,U_s(x),...,\partial^{\sigma}(U_s)(x)) \circ dS^{\alpha}_s}
\end{equation}
and we write
$$dU_t=\Psi_{\alpha}(U_t)F_{\alpha}(U_t) \circ dS^{\alpha}_t.$$
We call the SPDEs of the form \refeqn{equation_nonlocal} weakly local SPDEs.
\end{remark}

In order to reformulate Definition \ref{definition_SPDE} in terms of a standard SDE in $\jinf(M,N)$, we consider $U_t$ such that
 $U_t$ is a $C^{\infty}$ semimartingale for any $x \in N$  and we define a semimartingale $\bold{U}_t(x)$ taking values in $\jinf_{x}$ (the fiber of $x \in M$ in the bundle $\jinf(M,N)$) in the following way
$$u^i_{\sigma}(\bold{U}_t(x))=\partial^{\sigma}(U^i_t)(x).$$
Obviously  $\bold{U}$ is a semimartingale taking values in $\jinf(M,N)$ for any $x \in M$, i.e.  $f(x,\bold{U}_t(x))$ is a real semimartingale  for any $f \in \mathcal{F}$. Furthermore, fixing $t\in \mathbb{R}_+$ and $\omega \in \Omega$, the section $x \mapsto \bold{U}_t(x)(\omega)$ is an integral section of the Cartan distribution $\mathcal{C}$, so that
$$D_{x^i} \in T\mathcal{K}^{\bold{U}_t(\omega)}$$
for any $i=1,...,m$. The process $\bold{U}_t$ dependent on the parameter $x \in M$ is called the lifting of $U_t$ to $\jinf(M,N)$.\\
Conversely, if $\bold{P}_t(x)$ is a process dependent on $x \in M$ and taking values in $\jinf_x$ which is a semimartingale in $J(M,N)$ and, for $t \in \mathbb{R}_+$ and $\omega \in \Omega$,  $\bold{P}_t(\omega)$ is an integral manifold of the Cartan distribution, then there exists a function $U^{\bold{P}}$ such that the lifting $\bold{U}^{\bold{P}}$ of $U^{\bold{P}}$ to $\jinf(M,N)$ is exactly $\bold{P}$. This assertion can be proved using Theorem \ref{theorem_finite} and the fact that $U^{\bold{P},i}_t(x)=u^i \circ \bold{P}_t(x)$. \\

If $\mathcal{M}$ is a smooth manifold  and $Y_1,...,Y_r$ are $r$ vector fields on $\mathcal{M}$, the semimartingale $X$ on $\mathcal{M}$ is a solution to the SDE associated with $Y_1,...,Y_r$ and the semimartingales $S^1,...,S^r$ if and only if, for any $f \in \cinf(\mathcal{M})$,
$$f(X_t)-f(X_0)=\int_0^t{Y_{\alpha}(f)(X_s)\circ dS^{\alpha}_s}.$$
In the following, if $X$ is a solution to the SDE associated with $Y_1,...,Y_r$ and $S^1,...,S^r$, we write
$$dX_t=Y_{\alpha} \circ dS^{\alpha}_t.$$

\begin{theorem}\label{theorem_SPDE1}
The $C^{\infty}$ semimartingale $U$ is a solution to the SPDE associated with $F_1,...,F_r$ and $S^1,...,S^r$ if and only if, for any $x \in M$,
$\bold{U}_t(x)$ is a solution to the SDE associated with $V_{F_1},...,V_{F_r}$ and $S^{1},...,S^r$, i.e., for any $f \in \mathcal{F}$,
\begin{equation}\label{equation_SPDE1}
f(x,\bold{U}_t(x))-f(x,\bold{U}_0(x))=\int_0^t{V_{F_{\alpha}}(f)(x,\bold{U}_s(x))\circ dS^{\alpha}}.
\end{equation}
\end{theorem}

In order to prove Theorem \ref{theorem_SPDE1} we recall the following lemma.

\begin{lemma}\label{lemma_SPDE1}
If $G_t(x)$ is a $C^h$ semimartingale and $S$ is a real valued semimartingale, then
$$\int_0^{t}{G_s(x) \circ dS_s},$$
is a $C^{h-1}(M,N)$ semimartingale and, for any multi-index $|\sigma| < h$,
$$\partial^{\sigma}\left(\int_0^{t}{G_s(x) \circ dS_s} \right)=\int_0^t{\partial^{\sigma}(G_s)(x) \circ dS_s}.$$
\end{lemma}
\begin{proof}
The proof is given in \cite{Kunita1990} Exercise 3.1.6.
${}\hfill$\end{proof}\\

\noindent \begin{proof}[Proof of Theorem \ref{theorem_SPDE1}]
If $\bold{U}_t(x)$ is the solution to the SDE associated with $V_{F_1},...,V_{F_r}$ and $S^1,...,S^r$, then $U_t(x)=u \circ \bold{U}_t(x)$ is a
solution to the SPDE \refeqn{equation_definitionSPDE1}, since equation \refeqn{equation_SPDE1} becomes equation
\refeqn{equation_definitionSPDE2} if we choose $f=u^i$. \\
Conversely, if $U_t(x)$ is a solution to the SPDE \refeqn{equation_definitionSPDE1}, we have
\begin{eqnarray*}
\partial_{x^i}(F_{\alpha}(x,U_t(x),U_{\sigma,t}(x)))&=&
\partial_x^i(F_{\alpha})(x,U_t(x),U_{\sigma,t}(x))+\sum_{\sigma}\partial^{\sigma+1_i}(U_t)(x)\partial_{u^i_{\sigma}}(F_{\alpha})(x,U_t(x),U_{\sigma,t}(x))\\
&=&D_{x^i}(F_{\alpha})(x,U_t(x),U_{\sigma,t}(x)).
\end{eqnarray*}
By induction it is possible to prove
$$\partial^{\sigma}(F_{\alpha}(x,U_t(x),U_{\sigma,t}(x)))=D^{\sigma}(F_{\alpha})(x,U_t(x),U_{\sigma,t}(x))$$
and by Lemma \ref{lemma_SPDE1} we find
$$\partial^{\sigma}(U_t)(x)=\int_0^t{D^{\sigma}(F_{\alpha})(x,U_{s}(x),U_{\sigma,s}(x))\circ dS^{\alpha}_t}.$$
Using the It\^o formula for $x$ fixed and the previous equation we obtain the thesis.
${}\hfill$\end{proof}

\begin{remark}\label{remark_SPDE2}
Theorem \ref{theorem_SPDE1} can be extended to the case of more general SPDEs as described in Remark \ref{remark_SPDE1}. In this case
 the SDE solved by $\bold{U}_t(x)$ is
$$f(x,\bold{U}_t(x))-f(x,\bold{U}_0(x))=\int_0^t{\Psi_{\alpha}(U_s)V_{F_{\alpha}}(f)(x,\bold{U}_s(x))\circ dS^{\alpha}_s}.$$
This SDE  depends not only on $\bold{U}_t(x)$ but also on all the functions $U_t$, since the functional
$\Psi_{\alpha}$ is, in  general,  non-local.
\end{remark}

In the following we discuss the relationship between the notion  of solution to an SPDE  introduced by Definition  \ref{definition_SPDE} and
the usual definition given in terms of the theory of martingales taking values in Hilbert (or Banach) spaces (see, e.g. \cite{DaPrato1992}).\\

We start by considering the It\^o  reformulation of equation \refeqn{equation_definitionSPDE2},  which is simpler to use in the Hilbert space setting. By Theorem \ref{theorem_SPDE1} and using the relationship between Stratonovich and It\^o integral, we have that $U_t$ solves the SPDE associated with $F_1,...,F_r$ and $S^1,...,S^r$ in the sense of Definition \ref{definition_SPDE} if and only if
\begin{eqnarray*}
U^i_t(x)-U^i_0(x)&=&\int_0^t{F^i_{\alpha}(x,U_s(x),...,\partial^{\sigma}(U_s)(x))dS^{\alpha}_s}+\\
&&+\frac{1}{2}\int_0^t{V_{F_{\alpha}}(F^i_{\beta})(x,U_t(x),...,\partial^{\tau}(U_t)(x))d[S^{\alpha},S^{\beta}]}.
\end{eqnarray*}

\begin{definition}\label{definition_strong_weak1}
Let $H$ be an Hilbert space containing some subset of smooth functions defined on $M$. If $U^i_t(x) \in H$, we say that $U^1,...,U^n$ is a strong solution to the SPDE associated with
$F_1,...,F_r$ and $S^1,...,S^r$ if $F^i_{\alpha}(x,U_t,...,\partial^{\sigma}(U_t))$ and
$V_{F_{\alpha}}(F^i_{\beta})(x,U_t,...,\partial^{\tau}(U_t))$ are locally bounded processes in $H$ and
\begin{equation}\label{equation_strong}
U^i_t-U^i_0=\int_0^t{F^i_{\alpha}(x,U_s,...,\partial^{\sigma}(U_s)) dS^{\alpha}_s}+\frac{1}{2}\int_0^t{V_{F_{\alpha}}(F^i_{\beta})(x,U_t,...,\partial^{\tau}(U_t))d[S^{\alpha},S^{\beta}]_s},
\end{equation}
where the integrals are usual It\^o integrals in $H$. We say that $U_t$ is a weak solution to the SPDE associated with $F_1,...,F_r$ and
$S^1,...,S^r$ if, for any $\xi \in V \subset H$, where $V$ is a suitable subspace of $H$ which separates the points of $H$, $\langle \xi ,
F^i_{\alpha}(x,U_t,...,\partial^{\sigma}(U_t)) \rangle$ and $\langle \xi ,
V_{F_{\beta}}(F^i_{\alpha})(x,U_t,...,\partial^{\tau}(U_t)) \rangle$ are  real locally bounded processes and the following equality holds
\begin{equation}\label{equation_weak}
\langle \xi , U^i_t-U^i_0 \rangle = \int_0^t{\langle \xi, F^i_{\alpha}(x,U_s,...,\partial^{\sigma}(U_s)) \rangle dS^{\alpha}_s}+\frac{1}{2}\int_0^t{\langle \xi, V_{F_{\alpha}}(F^i_{\beta})(x,U_t,...,\partial^{\tau}(U_t)) \rangle d[S^{\alpha},S^{\beta}]_s},
\end{equation}
(here the integrals are usual $\mathbb{R}$ It\^o integrals).
\end{definition}

In general, it is not easy to find the relationship between the two notions of solution proposed in Definition \ref{definition_SPDE} and in Definition \ref{definition_strong_weak1}. For this reason we need to introduce an additional hypothesis (which is satisfied by the usual Hilbert spaces considered in SPDEs theory). Given a  smooth function  $f \in \cinf_0(M)$ with compact support, we  define a linear functional
$$\begin{array}{rccc}
l_f:&C^0(M) & \longrightarrow & \mathbb{R}\\
& g & \longmapsto & l_f(g):=\int_M{f(x)g(x)dx}. \end{array}$$
We say that the Hilbert space $H$ satisfies the hypotheses $L$ if
\begin{itemize}
\item there exists a subset  $\mathcal{L} \subset \cinf_0(M)$ such that, for any $f \in \mathcal{L}$, the functional $l_f:H \cap \cinf(M) \rightarrow \mathbb{R}$ can be extended in a unique continuous way to all $H$;\\
\item the functionals of the form $l_f$ for $f \in \mathcal{L}$ separate the points of $H$ and of the Fr\'echet space $C^0(M)$.
\end{itemize}
An example of  Hilbert spaces satisfying the hypothesis $L$ are Sobolev spaces $H^r_w(M)$ of function weakly derivable $r$ times and whose weak derivatives are square integrable with respect to the measure $w(x)dx$, where $w(x)$ is a positive continuous function $w:M \backslash \{x_1,...,x_l\} \rightarrow \mathbb{R}$ and  $x_1,...,x_l \in M$. In this case the set $\mathcal{L}$ is formed by the functions $f \in \cinf_0(M)$ which are identically zero in some neighborhood of $x_1,...,x_l$.

\begin{proposition}\label{proposition_strong}
Let $H$ be a Hilbert space satisfying the hypothesis $L$. If $U^i_t$ is a $\cinf$ semimartingale and  $U^i_t$,
$F^i_{\alpha}(x,U_s,...,\partial^{\sigma}(U_s))$ and $V_{F_{\alpha}}(F^i_{\beta})(x,U_t,...,\partial^{\tau}(U_t))$ are locally bounded processes
in $H$, then the notion of solutions to an SPDE given in Definition \ref{definition_SPDE} and the two definitions given in  Definition
\ref{definition_strong_weak1} are equivalent.
\end{proposition}
\begin{proof}
We prove that Definition \ref{definition_SPDE} is equivalent to the definition of weak solution in Definition  \ref{definition_strong_weak1}. The equivalence between weak and strong solutions under the hypotheses of the proposition is standard.\\
Suppose that $U^i_t$ is a solution to the SPDE $F_1,...,F_r$ and $S^1,...,S^r$ with respect to Definition \ref{definition_SPDE}. Fix $f \in \cinf_0(M)$ and denote by $K$ the support of $f$. Since $U^i_t$ are $\cinf$ semimartingales, $U^i_t$, $F^i_{\alpha}(x,U_s,...,\partial^{\sigma}(U_s))$ and $V_{\beta}(F^i_{\alpha})(x,U_s,...,\partial^{\sigma}(U_s))$ are locally bounded processes in $C^0(K)$. Definition \ref{definition_SPDE} is equivalent to say that for any Dirac delta distribution $\delta_x$ centred in $x \in K$, the following equality holds
\begin{equation}\label{equation_delta1}
\delta_x(U^i_t-U^i_0)=\int_0^t{\delta_x(F^i_{\alpha}(x,U_s,...,\partial^{\sigma}(U_s)))dS^{\alpha}_s}+\int_0^t{\delta_x(V_{F_{\alpha}}(F^i_{\beta})(x,U_t,...,\partial^{\tau}(U_t)))d[S^{\alpha},S^{\beta}]_s}.
\end{equation}
Since Dirac delta functionals are continuous linear functionals in $(C^0(K))^*$ which separate the points of $C^0(K)$ and $l_f$ is a continuous linear functional in $(C^0(K))^*$, there exists a succession $l_n \in (C^0(K))^*$, made by finite linear combinations of Dirac deltas, which converges $\text{weakly}^*$ to $l_f$ in $(C^0(K),(C^0(K))^*)$. Furthermore, since $U^i_t$, $F^i_{\alpha}(x,U_s,...,\partial^{\sigma}(U_s))$ and $V_{F_{\alpha}}(F^i_{\beta})(x,U_t,...,\partial^{\tau}(U_t))$ are locally bounded in $C^0(K)$ and $l_n$ is strongly bounded in $(C^0(K))^*$, there exists a locally bounded process $H_t$ in $\mathbb{R}_+$ such that
$$|l_n(U^i_t)|,|l_n(F^i_{\alpha}(x,U_s,...,\partial^{\sigma}(U_s)))|,|l_n(F^i_{\alpha}(x,U_s,...,\partial^{\sigma}(U_s))|\leq H_t,$$
almost surely. Equation \refeqn{equation_delta1} holds with $\delta_x$ replaced by $l_n$ since $l_n$ is a finite linear combination of Dirac deltas. Taking the limit for $n \rightarrow + \infty$, by the dominate convergence theorem for semimartingales (see \cite[Chapter IV, Theorem 32]{Protter1990}), we obtain
\begin{equation}\label{equation_delta2}
l_f(U^i_t-U^i_0)=\int_0^t{l_f(F^i_{\alpha}(x,U_s,...,\partial^{\sigma}(U_s)))dS^{\alpha}_s}+\frac{1}{2}\int_0^t{l_f(V_{F_{\alpha}}(F^i_{\beta})(x,U_t,...,\partial^{\tau}(U_t)))d[S^{\alpha},S^{\beta}]_s}.
\end{equation}
Using a similar reasoning and the fact that the linear space composed by $l_f$, with $f \in \mathcal{L}$, separates the points of $H$ and that $U^i_t$, $F^i_{\alpha}(x,U_s,...,\partial^{\sigma}(U_s))$ and $V_{F_{\alpha}}(F^i_{\beta})(x,U_t,...,\partial^{\tau}(U_t))$ are locally bounded in $H$, we obtain that equation \refeqn{equation_weak} holds for any $\xi \in H$  and thus $U^i_t$ is a weak solution. \\
Conversely, if $U^i_t$ is a weak solution to the SPDE associated with $F_1,...,F_r$ and $S^1,...,S^r$, since the space $V \subset H$ separates the point of $H$ and $U^i_t,...$ are locally bounded in $H$, it is possible to prove \refeqn{equation_delta2} for any $f \in \mathcal{L}$. Since the space composed by $l_f$, with $f \in \mathcal{L}$, separates the points of $C^0(M)$ and $U^i_t,...,$ are locally bounded in $C^0(K)$ for any compact set $K \subset M$, we can prove  \refeqn{equation_delta1} which is   equivalent to Definition \ref{definition_SPDE}.
${}\hfill$\end{proof}\\

\noindent In general, proving the local boundedness of the  $\cinf$ semimartingales  $U^i_t,...$, $F^i_{\alpha}(x,U_s,...,\partial^{\sigma}(U_s))$ and $V_{F_{\alpha}}(F^i_{\beta})(x,U_t,...,\partial^{\tau}(U_t))$  with respect to the norm of $H$ as required by Proposition \ref{proposition_strong} is quite hard. Nevertheless there is a case where verifying this hypothesis is trivial. If the closure $\bar{M}$ of $M$ in $\mathbb{R}^m$ is compact and $U^i_t$ is a $\cinf$ semimartingale on all $\bar{M}$ (in other words for any $x$ derivatives $\partial^{\sigma}(U^i_t)(x)$ there exists the finite limit $x \rightarrow x_0 \in \partial M$) the processes  $U^i_t$, $F^i_{\alpha}(x,U_s,...,\partial^{\sigma}(U_s))$ and $V_{F_{\alpha}}(F^i_{\beta})(x,U_t,...,\partial^{\tau}(U_t))$ are locally bounded in all the Sobolev spaces of the form $H^r_w(M)$, where $w \in L^1(M)$.

\subsection{Finite dimensional solutions to SPDEs}

\begin{definition}\label{definition_finitedimensional}
 A smooth function  $K:M \times B \rightarrow N$  is a finite dimensional solution to the
SPDE \refeqn{equation_definitionSPDE1} if, for any $b_0 \in B$, there exists a semimartingale $B_t$ taking values in $B$
such that $B_0=b_0$  and $K(x,B_t)$ is a solution to the SPDE \refeqn{equation_definitionSPDE1}.
\end{definition}

It is important to note that, if $B_t \in B$ is a semimartingale and $K$ is a smooth function, then $K(x,B_t)$ is a semimartingale dependent on
the parameter $x$ in the sense of Definition \ref{definition_semimartingale}. Indeed, if we fix $x \in M$, since the function $K(x,b)$ is smooth
in all its arguments, $K(x,B_t)$ is a semimartingale, being obtained transforming  the semimartingale $B_t \in B$ with respect to the $\cinf(B,N)$
function $b \mapsto K(x,b)$.

\begin{remark}\label{remark_strong_weak}
We can request that $K$ is a finite dimensional solution to the SPDE associated with $F_1,...,F_r$ and $S^1,...,S^r$ with respect to the weak
and strong definition of solution to an SPDE in a Hilbert space $H$ given in Definition \ref{definition_strong_weak1}. Thanks to Proposition
\ref{proposition_strong} all these definitions are equivalent whenever the function $K$ is locally bounded in $H$ i.e., for any compact $K_B$,
there exists a constant $C_{K_B}$ such that $\sup_{b \in K_B}|K(\cdot,b)| \leq C_{K_B}$.
\end{remark}

\begin{theorem}\label{theorem_SPDE2}
If, for any $\alpha=1,...,r$,   $V_{F_{\alpha}} \in \mathcal{K}^{K}$,  then $K$ is a finite dimensional solution to the SPDE
\refeqn{equation_definitionSPDE1}. Conversely, if  $K$ is a finite dimensional solution to  \refeqn{equation_definitionSPDE1}
and  $\frac{dA^{\alpha\beta}_t}{dt}=\frac{d[S^{\alpha},S^{\beta}]_t}{dt}$ is nonsingular for all $t \in \mathbb{R}_+$, then $V_{F_{\alpha}} \in
\mathcal{K}^K$.
\end{theorem}
\begin{proof}
If $V_{F_{\alpha}} \in \mathcal{K}^{K}$ there exist $r$ uniquely determined vector fields $Y_1,...,Y_r$ in the trivial bundle $M \times B$ such
that $\bold{K}_{*}(Y_{\alpha})=V_{F_{\alpha}}$. Since $V_{F_{\alpha}}$ are vertical in $\jinf(M,N)$, the vector fields $Y_{\alpha}$ are vertical
in the bundle $M \times B$. Furthermore, since $[D_{x^i},V_{F_{\alpha}}]=0$ and $\bold{K}_*(\partial_{x^i})=D_{x^i}$, we have that
$[\partial_{x^i},Y_{\alpha}]=0$ and so the vector fields $Y_{\alpha}$ are independent of $x^i$. If $B_t$ is the unique solution on $B$ to the
SDE associated with $Y_1,...,Y_r$ and $S^1,...,S^r$ such that $B_0=b_0 \in B$, then $U_t(x)=K(x,B_t)$ is a solution to the SPDE
\refeqn{equation_definitionSPDE1}. We prove this fact by showing that $\bold{U}_t(x)$ is a solution to the SDE $V_{F_{1}},...,V_{F_r}$ and
$S^1,...,S^r$ and then using Theorem \ref{theorem_SPDE1}. In fact, if $f \in \mathcal{F}$, then
\begin{eqnarray*}
f(x,\bold{U}_t(x))-f(x,\bold{U}_0(x))&=&\bold{K}^*(f)(x,B_t)-\bold{K}^*(f)(x,B_0)\\
&=&\int_0^t{Y_{\alpha}(\bold{K}^*(f))(x,B_s) \circ dS^{\alpha}_s}\\
&=&\int_0^t{\bold{K}^*(\bold{K}_*(Y_{\alpha})(f))(x,B_s) \circ dS^{\alpha}_s}\\
&=&\int_0^t{\bold{K}^*(V_{\alpha}(f))(x,B_s) \circ dS^{\alpha}_s}=\int_0^t{V_{\alpha}(f)(x,\bold{U}_s(x))\circ dS^{\alpha}_s}.
\end{eqnarray*}
Conversely, suppose that, for any $b_0 \in B$, there exists a semimartingale $B_t \in B$ such that $B_0=b_0$ and $K(x,B_t)$ is a solution to the
SPDE \refeqn{equation_definitionSPDE1}. If, for any function $f \in \mathcal{F}$ such that $f|_{\mathcal{K}^{K}}=0$, we have
$V_{F_{\alpha}}(f)|_{\mathcal{K}^K}=0$, then $V_{F_{\alpha}} \in T\mathcal{K}^{K}$. Let  $f \in \mathcal{F}$ be such that
$f|_{\mathcal{K}^{K}}=0$.  By It\^o formula we have
\begin{eqnarray*}
0&=&f(\bold{K}(x,B_t))-f(\bold{K}(x,B_0))\\
&=&\int_0^t{V_{F_{\alpha}}(f)(\bold{K}(x,B_s)) \circ dS^{\alpha}_s},
\end{eqnarray*}
and this ensures that the quadratic covariation of $\int_0^t{V_{F_{\alpha}}(f)(\bold{K}(x,B_s)) \circ dS^{\alpha}_s}$ with any $S^{\beta}$ is zero, i.e.
$$\int_0^t{V_{F_{\alpha}}(f)(\bold{K}(x,B_s))dA^{\alpha\beta}_s}=0.$$
Since the matrix $A_s$ is nonsigular for any $t$, in particular we  have that
$$V_{F_{\alpha}}(f)(\bold{K}(x,B_t))=0,$$
almost surely and for any $t \in \mathbb{R}_+$. Taking the limit for $t \rightarrow 0$ we obtain $V_{F_{\alpha}}(f)(\bold{K}(x,b_0))=0$. Since
$b_0 \in B$ is a generic point and $\bold{K}(x,b)$ is a surjective map from $M \times B$ into $\mathcal{K}^K$ we find
$V_{F_{\alpha}}(f)|_{\mathcal{K}^K}=0$.
${}\hfill$\end{proof}

\begin{remark}
It is possible to generalize Theorem \ref{theorem_SPDE2} in several directions. First it is possible to state both the sufficient
and necessary condition of Theorem \ref{theorem_SPDE2} for the SPDEs described in Remark \ref{remark_SPDE1}.\\
Furthermore, if we consider $B$ as a smooth manifold with boundary, the manifold $\mathcal{K}^K$ turns out to be itself a manifold with boundary and the sufficient
condition of Theorem \ref{theorem_SPDE2} is no more true. Indeed we have to add a new condition, i.e. that the SDE solved by the process $B_t$  should have a solution
for any starting point $b_0 \in B$. We remark that  this additional condition  is satisfied, for example, when $V_{F_{\alpha}} \in T (\partial \mathcal{K}^K)$.
\\
Finally, the condition on $A^{\alpha\beta}$ can be relaxed: in particular, if $S^{\alpha}$ (for $\alpha=1,...,l$) are absolutely
continuous, $\frac{dS^{\alpha}}{dt} \not =0$ almost surely and $\frac{dS^{\alpha}}{dt}$ are almost surely linearly independent with respect to
the time $t$,  we have to ensure that  $\frac{dA^{\alpha\beta}}{dt}$  is nonsingular only for $\alpha,\beta >l$. For example, this is the case
when $S^1_t=t$ and $S^{\alpha}_t=W^{\alpha}_t$ (for $\alpha=2,...,r$), where $W^{\alpha}_t$ are $r-1$ independent Brownian motions.
\end{remark}

It is interesting to note that Theorem \ref{theorem_SPDE2} provides an explicit method to construct the process $B_t$ appearing in  Definition \ref{definition_finitedimensional} when
 we do not  have the explicit reduction function $K(x,b)$ but only the finite dimensional manifold  $\mathcal{K}$. In fact, taking
a coordinate system $x^i$ on $M$ and  coordinates $z^1,...,z^p$ on $\mathcal{K}$ (we can use for example some $u^i,u^i_{\sigma}$), there are  some functions $\Xi^i_{\alpha}(x,z)$ such that
$$V_{F_{\alpha}}|_{\mathcal{K}}=\Xi^i_{\alpha}(x,z)\partial_{z^i},$$
the vector fields $V_{F_{\alpha}}$ being tangent to $\mathcal{K}$. On the other hand, the fact that  $\mathcal{K}$ is a canonical manifold ensures that also the vector fields $D_{x^i}$ are tangent to $\mathcal{K}$: in particular, there are some
functions $\Sigma^j_i(x,z)$ such that
$$D_{x^i}=\partial_{x^i}+\Sigma^j_i(x,z) \partial_{z^j}.$$
We define some processes $Z^i_t(x)$ which solve the following system of SDE in $t$ and PDE in $x^i$
\begin{eqnarray}
dZ^i_t(x)=\Xi^i_{\alpha}(x,Z_t(x)) \circ dS^{\alpha}_t\label{equation_solve1}\\
\partial_{x^i}(Z^j_t(x))=\Sigma^j_i(x,Z_t(x)).\label{equation_solve2}
\end{eqnarray}
The function $u^i \in \mathcal{F}$ can be expressed using the coordinates $(x^i,z^j)$ in $\mathcal{K}$, which means that there exists a function
$\mathfrak{U}^i(x,z)$ such that $u^i|_{\mathcal{K}}=\mathfrak{U}^i(x,z)$. With this notation the finite dimensional solution to the SPDE
\refeqn{equation_definitionSPDE1} is given by
$$U^i_t(x)=\mathfrak{U}^i(x,Z_t(x)).$$

\section{Characteristics and the construction of finite dimensional solutions to SPDEs}\label{section_characteristics}

In section \ref{section_finite} we reduced the problem of finding finite dimensional smooth solutions to an SPDE defined by $F_1,...,F_r \in \mathfrak{F}^n$ and the semimartingales $S^1,...,S^r$ to the problem of finding finite dimensional canonical manifolds $\mathcal{K}$ such that $V_{F_1},...,V_{F_r} \in T\mathcal{K}$. In this section we  provide a method to explicitly construct such canonical manifolds. The construction is based on the concept of characteristic of an evolution vector field, which is an extension of the usual notion of characteristic of a first order PDE.

\subsection{A necessary condition for the existence of finite dimensional solution to an SDPE}

In this section we provide a condition to be satisfied by the vector fields $V_{F_1},...,V_{F_r}$  in order to guarantee that a  SPDE defined by $F_1,...,F_r$ and $S^1,...,S^r$  admits a finite dimensional solution. In order to obtain an  extension of the necessary condition of Frobenious theorem,   we prove  the following

\begin{proposition}\label{proposition_necessary}
If the evolution vector fields $V_{F_1},...,V_{F_r}$ are in the tangent space of a finite dimensional manifold $\mathcal{K}$, then $V_{F_1},...,V_{F_r}$ generate a finite dimensional module on $\mathcal{K}$.
\end{proposition}
\begin{proof}
Being $V_{F_i},V_{F_j} \in T\mathcal{K}$, then $[V_{F_i},V_{F_j}]=V_{[F_i,F_j]} \in T\mathcal{K}$. Since $T\mathcal{K}$ is finite dimensional, $V_{F_1},...,V_{F_r}$ and all their Lie brackets form a finite dimensional module on $\mathcal{K}$.
\hfill \end{proof}\\

\noindent Using Proposition \ref{proposition_necessary} and the fact that the commutator of two evolution vector fields is an evolution vector field, we can suppose that $\mathcal{S}=\spann\{V_{F_1},...,V_{F_r}\}$ is a finite dimensional module on $\mathcal{K}$. Indeed, if is not the case,  we can add to the list of $V_{F_i}$ all their commutators $V_{[F_i,F_j]},...,V_{[F_i,[F_j,F_k]]},...$ and,  since $T\mathcal{K}$ is finite dimensional, we are sure that we are adding a finite number of vector fields. \\
In particular, if $V_{F_1},...,V_{F_r} \in T\mathcal{K}$,  we can suppose that $\mathcal{S}$ is a finite dimensional formally integrable module on $\mathcal{K}$. Since $V_{F_i}$ are not general vector fields on $\jinf(M,N)$ but they are evolution vector fields we can prove a stronger proposition.

\begin{proposition}\label{proposition_lie_algebra}
Let $V_{F_1},...,V_{F_r}$ be evolution vector fields in $\jinf(M,N)$ such that $\mathcal{S}$ is an $r$-dimensional (formally) integrable
distribution on a submanifold $\mathcal{K}$ of $\jinf(M,N)$. If
$$[V_{F_i},V_{F_j}]=\sum_h \lambda_{i,j}^hV_{F_h}$$
then $D_l(\lambda_{i,j}^h)=0$ on $\mathcal{K}$.
\end{proposition}
\begin{proof}
The proof is given  for the case $N=M=\mathbb{R}$ and $\mathcal{H}=\jinf(M,N)$;  the general case is a simple generalization of this one.\\
Since $\mathcal{S}$ is $r$-dimensional, for any point $p \in \jinf(M,N)$ there exist a neighborhood $U$ of $p$ and an integer $h \in
\mathbb{N}_0$ such that the matrix $A=(D_x^{h+j-1}(F_i))|_{i,j=1,...,s}$ is non-singular. Moreover, since the commutator of two evolution vector
fields is an evolution vector field, there exist some $F_{i,j} \in \mathcal{F}$ such that $[V_{F_i},V_{F_j}]=V_{F_{i,j}}$ and, by the definition
of evolution vector field,  we have
\begin{equation}\label{CN1}
D^r_x(F_{i,j})=\sum_h \lambda^h_{i,j} D^r_x(F_h).
\end{equation}
Deriving with respect to $x$ the previous relations we obtain
\begin{equation}
\label{CN2}D^{r+1}_x(F_{i,j})=\sum_h D_x(\lambda^h_{i,j}) D^r_x(F_h)+\sum_h \lambda^h_{i,j} D^{r+1}_x(F_h)
\end{equation}
and combining \eqref{CN1} and \eqref{CN2} we find
$$\sum_h D_x(\lambda^h_{i,j}) D^r_x(F_h)=0.$$
Since the matrix $A$ is non-singular, we get  $D_x(\lambda^h_{i,j})=0$.
\hfill\end{proof}\\

\noindent Proposition \ref{proposition_lie_algebra} implies  that if an SPDE associated with $(F_1, \ldots, F_r)$ admits a finite dimensional solution  passing through any point of $\jinf(M,N)$, the vector fields $V_{F_1},...,V_{F_r}$ have to form not only a module on $\jinf(M,N)$ but a Lie algebra.
For this reason in the following we always suppose that $V_{F_1},...,V_{F_r}$ form a Lie algebra, i.e.  there exist some constants $\lambda^i_{j,k} \in \mathbb{R}$ such that
$$[V_{F_i},V_{F_j}]=\sum_h \lambda_{i,j}^hV_{F_h}.$$

\subsection{Characteristics for evolution vector fields}

In the previous section we provided an infinite dimensional analogous of necessary condition of Frobenious theorem.
Unfortunately the sufficient part of Frobenious theorem rests on  the existence of the flow of a smooth vector field on a finite dimensional manifold. Since $\jinf(M,N)$ is infinite dimensional, and  a general evolution vector field does not have flow, the necessary conditions given in  Propositions \ref{proposition_necessary} and  \ref{proposition_lie_algebra} are not sufficient. \\
Therefore, in order to find a sort of complete Frobenius theorem in our setting,  we have to introduce the concept of characteristic of an evolution vector field in $\jinf(M,N)$. We say that a vector field $V=\phi^i \partial_{x^i}+\psi^j \partial_{u^j}+\psi^j_{\sigma} \partial_{u^j_{\sigma}}$, where $\phi^i,\psi^j,\psi^j_{\sigma} \in \mathfrak{F}$, \emph{admits flow in $ \jinf(M,N)$ } if  there exists a sequence of smooth functions $\Phi_a^{x^i},\Phi_a^{u^j},\Phi_a^{u^j_{\sigma}} \in \mathfrak{F}$ (where $a$ is a parameter) such that
\begin{eqnarray*}
\partial_a(\Phi_a^{x^i})(x,u,u_{\sigma})&=&\phi^i(\Phi^{x^i}_a(x,u,u_{\sigma}),\Phi_a^{u^j_{\sigma}}(x,u,..),...)\\
\Phi_0^{x^i}(x,u,u_{\sigma})&=&x^i\\
\partial_a(\Phi_a^{u^j})(x,u,u_{\sigma})&=&\psi^j(\Phi^{x^i}_a(x,u,u_{\sigma}),\Phi_a^{u^j_{\sigma}}(x,u,..),...)\\
\Phi_0^{u^j}(x,u,u_{\sigma})&=&u^j\\
&...&
\end{eqnarray*}
Usually we can request that the functions $\Phi_a^{x^i},...,$ are defined only locally, which means that, for any open bounded interval $ I
\subset \mathbb{R}$  with $0 \in I$,  if $a \in I$, the functions $\Phi^{x^i}_a,...$ are defined in a suitable open subset of $\jinf(M,N)$ (see
\cite{DMevolution} for the details). The function $\Phi_a=(\Phi^{x^i}_a,\Phi^{u^j}_a,...):\jinf(M,N) \rightarrow \jinf(M,N)$ can be seen as a
smooth map between $\jinf(M,N)$ and itself. Indeed $\Phi_a$ has the property that, whenever $F \in \mathfrak{F}$,
$$\Phi_a^*(F)(x,u,u_{\sigma})=F(\Phi^{x^i}_a(x,u,u_{\sigma}),\Phi^{u^j}_a(x,...),...) \in \mathfrak{F}.$$
Furthermore, as in the finite dimensional case, we have
$$\partial_a(\Phi^*_a(F))=\Phi^*_a(V(F)).$$

\begin{definition}\label{definition_characteristic}
A function $F \in \mathfrak{F}^n$, defining the evolution vector field $V_F$, admits characteristics if there exist suitable  functions $h^1,...,h^m \in \mathfrak{F}$ such that the vector field
$$\bar{V}_F=V_F-\sum_{i=1}^m h^i D_i$$
admits flow in $\jinf(M,N)$. We call the vector field $\bar{V}_F$ the characteristic vector field of $F$ (or equivalently of $V_F$).
\end{definition}

The previous definition  provides a generalization  of the well known concept of characteristics for a
first order function $F \in \mathfrak{F}^1$ when $N=\mathbb{R}$ to the case  of $M,N$ of general dimension. Since in this paper we are mainly interested in the computational aspects
of the theory, in the following we provide an explicit example and we refer the interested reader to \cite{DMevolution}
where Definition \ref{definition_characteristic} is introduced and  studied in detail.  \\

Consider $M=N=\mathbb{R}$ and the function $F=u u_x$. The function $F \in \mathfrak{F}$ admits characteristic which are given by the vector field $\bar{V}_F$ defined as
$$\bar{V}_F=V_F-uD_x.$$
Indeed, the components of the vector field $\bar{V}_F$ are
$$\bar{V}_F=-u\partial_x+u_x^2\partial_{u_x}+2u_{xx}u_x\partial_{u_{xx}}+...$$
and it is simple to verify that, since  $\bar{V}_F(\mathfrak{F}^k) \subset \mathfrak{F}^k$ for any  $k$,  $\bar{V}_F$ admits flow which is given by
\begin{eqnarray*}
\Phi_a^x(x,u,u_{(n)})&=&x-au\\
\Phi_a^u(x,u,u_{(n)})&=&u\\
\Phi^{u_x}_a(x,u,u_{(n)})&=&\frac{u_x}{1-au_x}\\
\Phi^{u_{xx}}_a(x,u,u_{(n)})&=&\frac{u_{xx}}{(1-au_x)^2}\\
&...&\\
\Phi^{u_{(n)}}_a(x,u,u_{\sigma})&=&\frac{D_x(\Phi^{u_{(n-1)}}_a(x,u,u_{(n)}))}{D_x(\Phi^x_a(x,u,u_{(n)}))}\\
&...&
\end{eqnarray*}

\subsection{Building finite dimensional solutions to SPDEs}\label{Building differential constraints}

We start this section by introducing some technical notions in order to  prove two different versions of Frobinious theorem based on the concept of characteristics for evolution vector fields in $\jinf(M,N)$. \\

Let  $\mathcal{H}\subset \jinf(M,N) $ be a submanifold and $U$ be an open neighborhood of $p \in \mathcal{H}$. Given a sequence of independent functions $f^i \in \mathfrak{F}|_U$ ($i \in \mathbb{N}$) such that $\mathcal{H} \cap U$ is the set of zeros of $f^i$, we say that a distributions  $\Delta= \spann\{V_{G_1}, \ldots, V_{G_h}\}$ is transversal to $\mathcal{H}$ in $U$ if there exist $r_1, \ldots , r_h$ such that
the matrix $(\bar{V}_{G_i}(f^{r_j}))|_{i, j=1,...,h}$ has maximal rank in $U$.\\

Let $\mathfrak{G}_0 \subset \mathfrak{F}$ be a subalgebra of $\mathfrak{F}$ which can be generated, through the composition with smooth functions, by a finite set of elements. Starting from $\mathfrak{G}_0$ we  define a filtration $\mathfrak{G}_0 \subset \mathfrak{G}_1 \subset \mathfrak{G}_2 \subset ... \subset \mathfrak{F}$, where  the subalgebra $\mathfrak{G}_k$ is generated by the composition with smooth functions in $\mathfrak{G}_{k-1}$ and in $D_i(\mathfrak{G}_{k-1})$. We say that $\mathfrak{G}_k$ is a filtration  generating $\mathfrak{F}$ if $\mathfrak{F}=\cup_{k \in \mathbb{N}_0} \mathfrak{G}_k$. A natural  example of a filtration  generating $\mathfrak{F}$ is given by $\mathfrak{G}_k=\mathfrak{F}_k$. Let $G_1,...,G_h$ be a set of functions in $\mathfrak{F}^n$  admitting characteristic and let  $\bar{V}_{G_1},...,\bar{V}_{G_h}$ be the corresponding characteristic vector fields. We say that $G_1,...,G_h$ admits a common filtration if there exists a filtration $\mathfrak{G}_k$ constructed as above which generates $\mathfrak{F}$  and is such that $\bar{V}_{G_i}(\mathfrak{G}_k) \subset \mathfrak{G}_k$. \\

Finally, in order to simplify the notations, if   $\Phi^1_{a_1},...,\Phi^h_{a_h}$ are flows in $\jinf(M,N)$, for any $p \in \jinf(M,N)$ and $F \in \mathfrak{F}$, we write
\begin{eqnarray*}
\bold{\Phi}_{\bold{a}}(p)&=&\Phi^h_{a_h}(\Phi^{h-1}_{a_{h-1}}(...(\Phi^1_{a_1}(p))...))\\
\bold{\Phi}^*_{\bold{a}}(F)&=&\Phi^{1*}_{a_1}(\Phi^{2*}_{a_2}(...(\Phi^{h*}_{a_h}(F))...)),
\end{eqnarray*}
where $\bold{a}=(a_1,...,a_h)$.

\begin{theorem}\label{theorem_main1}
Let $F_1,...,F_l,G_1,...,G_h \in \mathfrak{F}^n$ generate a finite dimensional Lie algebra, and $G_1,...,G_h$ generate a finite dimensional Lie subalgebra such that $\dim(\spann\{V_F,V_{G_1},...,V_{G_{h}}\})=h+l$. If  $G_1,...,G_h$ admit characteristics and have a common filtration and  $\mathcal{H}$ is a finite dimensional canonical submanifold of $\jinf(M,N)$ such that $V_{F_1},...,V_{F_l} \in T \mathcal{H}$ and that $V_{G_1},...,V_{G_h}$ generate a  distribution transversal to $\mathcal{H}$, then there exists a suitable  neighborhood $\mathcal{V} \subset \mathbb{R}^h$  of $0$ such that
\begin{equation}\label{manifold kappa}
\mathcal{K}=\bigcup_{\bold{a} \in \mathcal{V}} \Phi^h_{a^h}(...(\Phi^{1}_{a^{1}}(\mathcal{H}))...),
\end{equation}
 is a finite dimensional canonical submanifold of $\jinf(M,N)$ and $V_{F_i},V_{G_j} \in T\mathcal{K}$.
\end{theorem}
\begin{proof}
The proof can be found in \cite{DMevolution}.
\hfill \end{proof}

\begin{theorem}\label{theorem_main2}
In the hypotheses and with the notations of Theorem \ref{theorem_main1} except that $F_1,...,F_l,G_1,...,G_h$ generate a finite dimensional Lie algebra, if $F_i,G_j$ are real analytic,  $\mathcal{H}$ is defined by real analytic functions and, denoting by $L=\langle  F_1,...,F_l, G_1, \ldots, G_h \rangle$ the Lie algebra generated by $F_i$ and $G_j$, we have
$$L|_{\mathcal{H}}  \subset T\mathcal{H} \oplus \spann\{V_{G_1},...,V_{G_h}\},$$
then $V_{F_i},V_{G_j} \in T\mathcal{K}$.
\end{theorem}
\begin{proof}
The proof can be found in \cite{DMevolution}.
 \hfill\end{proof}

\begin{remark}
If $\bar{V}_{G_i},V_F$ are real analytic and $\mathcal{H}$ is defined by real analytic equations, Theorem \ref{theorem_main1} implies Theorem \ref{theorem_main2}. On the other hand  Theorem \ref{theorem_main1} turns out to  be very useful  when we consider  smooth (not analytic) invariant manifolds $\mathcal{H}$.
\end{remark}

\begin{remark}
It is important to note that Theorems \ref{theorem_main1} and \ref{theorem_main2} hold also if $\mathcal{H}$ is a manifold with boundary.
In this case if $V_{G_{1}},...,V_{G_{h}} \in T(\partial\mathcal{H} )$  we obtain that $\mathcal{K}$ is also a local manifold with boundary.
\end{remark}

\section{A general algorithm to compute solutions to SPDEs}\label{section_general_algorithm}

In this section, starting from Theorem \ref{theorem_main1} and Theorem \ref{theorem_main2},  we provide a general algorithm to explicitly compute the finite dimensional solution to an SPDE. The main tool is   the introduction of a special coordinate system on the manifold $\mathcal{K}$ which permits to avoid most of the computational  problems in $\jinf(M,N)$. \\

Given an invariant submanifold $\mathcal{H}$ such that $V_{F_1},...,V_{F_l} \in T\mathcal{H}$, we have to  compute the characteristic flows $\Phi^1_{a^1},...,\Phi^h_{a^h}$ of $G_1,...,G_h$  in order to obtain $\mathcal{K}=\bold{\Phi}_{(a^1,...,a^h)}(\mathcal{H})$. Once we have $\mathcal{K}$, which by Theorem \ref{theorem_main1} and Theorem \ref{theorem_main2} is a finite dimensional solution to the SPDE defined by $F_1,...,F_l,G_1,...,G_h$, we can choose a coordinate system on $\mathcal{K}$ of the form $x^1,...,x^l,y^1,..,y^k$ and compute the explicit expressions for the vector fields $V_{F_1},...,V_{F_l},V_{G_1},...,V_{G_h}$ and $D_1,...,D_m$ in the coordinate system $(x,y)$. Finally, by solving   equations \refeqn{equation_solve1} and \refeqn{equation_solve2}, we obtain the explicit solution to the original SPDE. \\
In the general case it is not possible to explicitly perform all the described steps, so that it is not  possible to explicitly reduce the SPDE to a finite dimensional SDE.
Despite this fact, there are at least two cases where this reduction can be done:

\begin{itemize}
\item Case 1: the SPDE is defined by some functions $G_1,...,G_h$  admitting characteristics and  forming a finite dimensional Lie algebra \item
Case 2: the SPDE is defined by a function $F$ which does not admit characteristics and some functions $G_1,...,G_r$ which admit characteristics.
\end{itemize}

Furthermore, in order to explicitly compute the solution, we require two additional hypotheses
\begin{itemize}
\item the characteristics of $G_1,...,G_h$ admit a common filtration $\mathfrak{G}_0$ and the characteristic flow of $G_1,...,G_h$ can be explicitly computed,
\item we are able to solve the equation
$$\partial_a(f(x,a))=F(x,f(x,a),\partial^{\sigma}(f(x,a)))$$
for all $a \geq 0$ and for some initial condition $f(x,0)=f_0(x) \in \cinf(M,N)$.
\end{itemize}

All the previous hypotheses are generally satisfied  in the literature of finite dimensional solutions to SPDEs and they hold for all the examples in Section \ref{section_examplesbis} (the only exception is the second part of Section \ref{subsection_filtering}, where we consider an SPDE such that  Theorem \ref{theorem_main1} and Theorem \ref{theorem_main2} do not apply).

In Case 1 the first step consists in choosing the manifold $\mathcal{H}$ as the zeros of the following functions
$$h^i_{\sigma}=u^i_{\sigma}-\partial^{\sigma}(f^i(x)),$$
where $f^i \in \cinf(M,N)$. It is easy to check that $\mathcal{H}$ is a canonical submanifold of $\jinf(M,N)$, since
$$T\mathcal{H}=\spann\{D_1,...,D_m\}.$$
In order to apply Theorem \ref{theorem_main1} we need that, for any $x_0 \in M$, there exists a set of $h$  multi-indices
$\sigma^1,...,\sigma^h$ and of indices $i^1,...,i^h \in \{1,...,n\}$ such that
\begin{equation}\label{equation_degenerate}
\left(
\begin{array}{ccc}
D^{\sigma_1}(G^{i_1}_1)(x,f(x),\partial^{\tau}(f)(x)) & ... & D^{\sigma^h}(G^{i_1}_h)(x,f(x),\partial^{\tau}(f)(x))\\
... & ... & ...\\
D^{\sigma_h}(G^{i_h}_1)(x,f(x),\partial^{\tau}(f)(x)) & ... &D^{\sigma^h}(G^{i_h}_h)(x,f(x),\partial^{\tau}(f)(x))\end{array} \right)
\end{equation} has maximal rank. If $f^1,...,f^r$ are   real analytic functions, it is enough to check that  previous condition holds in one point or for generic smooth functions
$f^1,...,f^r$.\\
Under this hypothesis, we define the manifold $\mathcal{K}$ as in Theorem \ref{theorem_main1}
$$\mathcal{K}=\bigcup_{\bold{a} \in \mathcal{V}} \Phi^h_{a^h}(...(\Phi^{1}_{a^{1}}(\mathcal{H}))...).$$
This means that, for any $\bold{a}=(a_1,...,a_h)$ in a suitable neighborhood of the origin of $\mathbb{R}^h$,   $\mathcal{K}$ is the set of all the points $p=(x,u,u_{\sigma}) \in \jinf(M,N)$ such that there exist $(a^1_p,...,a^h_p) \in \mathbb{R}^h$ satisfying
\begin{equation}\label{equation_algorithm1}
\bold{\Phi}^*_{(a^1_p,...,a^h_p)}(u^i_{\sigma})-\partial^{\sigma}(f^j)(\bold{\Phi}^x_{(a^1_p,...,a^h_p)}(x,u,u_{\sigma}))=0.
\end{equation}
We define a special set of functions, which we still denote  by $a^1(x,u,u_{\sigma}),...,a^h(x,u,u_{\sigma})$, satisfying
\begin{equation}\label{equation_algorithm2}
\bold{\Phi}^*_{(a^1(x,u,u_{\sigma}),...,a^h(x,u,u_{\sigma}))}(u^{i_k}_{\sigma^{k}})-
\partial^{\sigma^k}(f^{i_k})(\bold{\Phi}^x_{(a^1(x,u,u_{\sigma}),...,a^h(x,u,u_{\sigma}))}(x,u,u_{\sigma})=0.
\end{equation}
Hereafter, in order to avoid  confusion, we write $a^i$ only for the functions defined by equation \refeqn{equation_algorithm2}, while we use other letters, for example $\bold{b}=(b^1,...,b^h)$  to describe  the flow $\bold{\Phi}$ evaluated at some fixed  $\mathbf{b}\in \mathbb{R}^h$.   \\
Our regularity assumption on the matrix \eqref{equation_degenerate}
ensures that equation \refeqn{equation_algorithm1} has a unique local solution in a neighborhood of
$\pi^{-1}(x_0)$ and the functions $x^1,...,x^n,a^1,...,a^h$ provide  a local coordinate system for
$\mathcal{K}$ in a neighborhood of $\pi^{-1}(x_0)$. Indeed, using  \refeqn{equation_algorithm2},  we have that $\mathcal{K}$ is the set of zeros of
\begin{equation}\label{equation_K}
\left.\bold{\Phi}^*_{(a^1(x,u,u_{\sigma)},...,a^h(x,u,u_{\sigma}))}(u^i_{\sigma})-\partial^{\sigma}(f^i)(\bold{\Phi}^x_{(a^1(x,u,u_{\sigma)},...,a^h(x,u,u_{\sigma}))}(x,u,u_{\sigma}))\right|_{\mathcal{K}}=0.
\end{equation}
In this coordinate system the vector fields $V_{G_1},...,V_{G_h}$ have a special form, as showed  by the following

\begin{theorem}\label{theorem_algorithm}
The vector fields $V_{G_1}|_{\mathcal{K}},...,V_{G_h}|_{\mathcal{K}}$ satisfy the relations
$$V_{G_i}=\phi^l_i(a^1,...,a^h)\partial_{a^l},$$
and the smooth functions $\phi^l_i$ are such that
\begin{equation}\label{equation_algorithm3}
\phi^l_i(a^1,...,a^h)=-\delta^l_i \ \ \  \text{when }a^1,...,a^{i-1}=0
\end{equation}
\begin{equation}\label{equation_algorithm4}
\partial_{a^k}(\phi^l_i(a^1,...,a^h))=-\lambda_{k,i}^p \phi^l_p-\sum_{r > k} \phi^r_i \partial_{a^r}(\phi^l_k(a^1,...,a^h))
\ \ \ \text{when }a^1,...,a^{k-1}=0\text{ and }k \geq i,
\end{equation}
where
$$[G_i,G_j]=\lambda^k_{i,j}G_k.$$
\end{theorem}

\begin{lemma}\label{lemma_algorithm1}
If $a^1,...,a^h$ are defined by \refeqn{equation_algorithm2},  then
$$D_i(a^j)|_{\mathcal{K}}=0.$$
\end{lemma}
\begin{proof}
Using equations \refeqn{equation_algorithm2} defining the functions $a^i$ we have that
\begin{equation}
\label{equation_lemma_algorithm1}0=D_i(\bold{\Phi}^*_{(a^1,...,a^h)}(h^{i_k}_{\sigma^k}))=
\left.\left[D_i(a^j)\partial_{b^j}(\bold{\Phi}^*_{(b^1,...,b^h)}(h^{i_k}_{\sigma^k}))+D_i(\bold{\Phi}^*_{(b^1,...,b^k)}(h^{i_k}_{\sigma^k}))\right]
\right|_{b^1=a^1,...,b^h=a^h},
\end{equation} where
$$h^{i_k}_{\sigma^k}=u^{i_k}_{\sigma^k}-\partial^{\sigma^k}(f^{i_k})(x).$$
By  Theorem 7.1 of \cite{DMevolution} there exist suitable smooth functions $B^i_j(b^1,...,b^h,x,u,u_{\sigma})$ such that
$$\bold{\Phi}_{(b^1,...,b^h),*}(D_j):=\Phi^h_{b^h,*}(...(\Phi^1_{b^1,*}(D_j))...)=B^k_j D_k.$$
So equations \refeqn{equation_lemma_algorithm1} restricted on $\mathcal{K}$ become
\begin{eqnarray*}
0&=&D_i(a^j)\left.\left[\partial_{b^j}(\bold{\Phi}^*_{(b^1,...,b^k)}(h^{i_k}_{\sigma^k}))\right]\right|_{b^1=a^1,...,b^h=a^h}+\\
&&+\left.\left\{\bold{\Phi}^*_{(b^1,...,b^h)}\left[\bold{\Phi}_{(b^1,...,b^h),*}(D_i)(h^{i_k}_{\sigma^k})\right] \right\}
\right|_{b^1=a^1,...,b^h=a^h},\\
&=&D_i(a^j)\left.\left[\partial_{b^j}(\bold{\Phi}^*_{(b^1,...,b^k)}(h^{i_k}_{\sigma^k}))\right]\right|_{b^1=a^1,...,b^h=a^h}+\\
&&+\left.\left[\bold{\Phi}^*_{(b^1,...,b^h)}(B_i^r)\bold{\Phi}^*_{(b^1,...,b^h)}(D_r(h^{i_{k}}_{\sigma^k}))\right]\right|_{b^1=a^1,...,b^h=a^h}\\
&=&D_i(a^j)\left.\left[\partial_{b^j}(\bold{\Phi}^*_{(b^1,...,b^k)}(h^{i_k}_{\sigma^k}))\right]\right|_{b^1=a^1,...,b^h=a^h},
\end{eqnarray*}
where we use  relations \refeqn{equation_K}. Since
$\left.\left[\partial_{b^j}(\bold{\Phi}^*_{(b^1,...,b^k)}(h^{i_k}_{\sigma^k}))\right]\right|_{b^1=a^1,...,b^h=a^h}$ is nonsingular in a
neighborhood of $x_0$, we have that $D_i(a^j)|_{\mathcal{K}}=0$.
${}\hfill$\end{proof}\\

\noindent \begin{proof}[Proof of Theorem \ref{theorem_algorithm}]
Since $V_{G_i}(x^j)=0$, then $V_{G_i}|_{\mathcal{K}}=\phi^l_i(x,a^1,...,a^h)\partial_{a^i}$. If equations
\refeqn{equation_algorithm3} and \refeqn{equation_algorithm4} hold, then the functions $\phi^l_i$ are independent from $x$. So, in order  to
prove Theorem \ref{theorem_algorithm}, we need only  to prove   \refeqn{equation_algorithm3} and
\refeqn{equation_algorithm4}.\\
By Lemma \ref{lemma_algorithm1} we get
$$0=\bar{V}_{G_i}(\bold{\Phi}^*_{a^1,...,a^h}(h^{i_k}_{\sigma^k}))|_{\mathcal{K}}=\left.\left[V_{G_i}(a^j)\partial_{b^j}
(\bold{\Phi}^*_{(b^1,...,b^h)}(h^{i_k}_{\sigma^k}))+\bar{V}_{G_i}(\bold{\Phi}^*_{(b^1,...,b^h)}(h^{i_k}_{\sigma^k}))\right]\right|_{b^1=a^1,...,b^h=a^h,\mathcal{K}}
.$$ If $b^1,...,b^{i-1}=0$ we have that
$$\bar{V}_{G_i}(\bold{\Phi}^*_{(b^1,...,b^h)}(h^{i_k}_{\sigma^k}))=\partial_{b^i}(\bold{\Phi}^*_{(b^1,...,b^h)}(h^{i_k}_{\sigma^k})),$$
and we obtain  condition \refeqn{equation_algorithm3}. Equation \refeqn{equation_algorithm4} follows using the  definition of Lie brackets between vector fields and the Lie algebra structure of $G_1,...,G_h$
together with equation \refeqn{equation_algorithm3}.
${}\hfill$\end{proof}\\

\noindent Theorem \ref{theorem_algorithm} allows us  to compute the expressions of $V_{G_1},...,V_{G_h}$ on $\mathcal{K}$ in the coordinate system $(x^1,...,x^m,a^1,...,a^h)$, even if we are not able to compute the explicit expression of $a^1,...,a^h$. Indeed, equations \refeqn{equation_algorithm3} and \refeqn{equation_algorithm4} not only uniquely determine the functions $\phi^i_j$, but also permit to get their explicit expressions. In order to show this last assertion we propose here an example that will also be useful in Section \ref{section_examplesbis}.\\
Taking $M=N=\mathbb{R}$, $G_1=1,G_2=u,G_3=u^2$ and considering $\mathcal{H},\mathcal{K}$ as in the previous discussion, we have that
$$[G_1,G_2]=G_1, \ \ \ \ [G_1,G_3]=2G_2, \ \ \ \ [G_2,G_3]=G_3.$$
The equations for $\phi^i_1$ and $\partial_{a^1}(\phi^i_j)$ are
\begin{eqnarray*}
&\phi^1_1=-1, \ \ \ \ \phi^2_1=0, \ \ \ \ \phi^3_1=0,&\\
&\partial_{a^1}(\phi^1_2)=1, \ \ \ \ \partial_{a^1}(\phi^2_2)= 0, \ \ \ \ \partial_{a^1}(\phi^3_2)=0,&\\
&\partial_{a^1}(\phi^1_3)=-2\phi^1_2, \ \ \ \ \partial_{a^1}(\phi^2_3)=-2\phi^2_2, \ \ \ \ \partial_{a^1}(\phi^3_3)=-2\phi^3_2, &
\end{eqnarray*}
from which we obtain
\begin{eqnarray*}
&\phi^1_2=a^1+\tilde{f}^1(a^2,a^3), \ \ \ \ \phi^2_2=\tilde{f}^2(a^2,a^3), \ \ \ \  \phi^3_2=\tilde{f}^3(a^2,a^3),&\\
&\phi^1_3=-(a^1)^2-2a^1\tilde{f}^1+\tilde{g}^1(a^2,a^3), \ \ \ \phi^2_3=-2a^1\tilde{f}^2+\tilde{g}^2(a^2,a^3), \ \ \
\phi^3_3=-2a^1\tilde{f}^3+\tilde{g}^3(a^2,a^3).&
\end{eqnarray*}
From the equations of $\phi^i_2$ on $a^1=0$ and $\partial_{a^2}(\phi^i_3)$ we have
\begin{eqnarray*}
& \tilde{f}^1=0, \ \ \ \ \tilde{f}^2=-1, \ \ \ \ \tilde{f}^3=0 & \\
&\partial_{a^2}(\tilde{g}^1)=-\tilde{g}^1, \ \ \ \ \partial_{a^2}(\tilde{g}^2)=-\tilde{g}^2, \ \ \ \ \partial_{a^2}(\tilde{g}^3)=-\tilde{g}^3. &
\end{eqnarray*}
Solving the previous equations and imposing  $\tilde{g}^i=-\delta^{i}_3$ we get
$$\tilde{g}^1=0, \ \ \ \ \tilde{g}^2=0, \ \ \ \ \tilde{g}^3=-e^{-a^2}, $$
so that we find
\begin{eqnarray*}
V_{G_1}|_{\mathcal{K}}&=&-\partial_{a^1}\\
V_{G_2}|_{\mathcal{K}}&=&a^1\partial_{a^1}-\partial_{a^2}\\
V_{G_3}|_{\mathcal{K}}&=&-(a^1)^2\partial_{a^1}+2a^1\partial_{a^2}-e^{-a^2}\partial_{a^3}.
\end{eqnarray*}
We remark that we have been able to obtain the expressions  of $V_{G_i}$ without information on the manifold $\mathcal{K}$. This fact is a strong consequence of the Lie algebra structure of $G_i$.\\

Once we have the expressions of $\phi^i_j$, we can explicitly compute the finite dimensional SDE related to our SPDE, which is
\begin{equation}\label{equation_algorithm5}
dA^i_t=\phi^i_{\alpha}(A^1_t,...,A^h_t) \circ dS^{\alpha}_t,
\end{equation}
where the semimartingales $S^{\alpha}$ are the same ones of equation \refeqn{equation_definitionSPDE1}, and, if we know  the processes $A^1,...,A^h$, we can explicitly compute the solutions $U^1_t(x),...,U^n_t(x)$. Note that the hypothesis that $G_1,...,G_h$ admit a common filtration plays an important role. In fact, if we project the manifold $\mathcal{H}$ on the manifold defined by the algebra $\mathfrak{G}_0$, we find a finite set of functions $H_1,...,H_l$ such that $H_i \in \mathfrak{G}_0$ and $H_i(p)=0$ if and only if $p \in \mathcal{H}$. Since $\bar{V}_{G_i}(H) \in \mathfrak{G}_0$ we have that, for any $b^1,...,b^h \in \mathbb{R}^h$, $\bold{\Phi}_{(b^1,...,b^h)}(H^i) \in \mathfrak{G}_0$. Therefore,  the solution $U_t(x) \in \cinf(M,N)$  it is the unique smooth function such that
\begin{equation}\label{equation_algorithm7}
\bold{\Phi}_{(A^1_t,...,A^h_t)}(H^i)(x,U_t(x),\partial^{\sigma}(U_t(x)))=0.
\end{equation}
The previous set of equations completely determines the function $U_t(x)$. In the particular case  $\mathfrak{G}_0=\mathfrak{F}_0$ (the set of smooth functions which depend only on $x^i,u^j$) we have that $\bold{\Phi}^{x^i}_{(b^1,...,b^h)}$ and $\bold{\Phi}^{u^j}_{(b^1,...,b^h)}$ depend only on $x$ and $u$; this means that $U^i_t(x)$ is the unique solution to the equations
\begin{equation}\label{equation_algorithm6}
\bold{\Phi}^{u^i}_{(A^1_t,...,A^h_t)}(x,U_t(x))-f(\bold{\Phi}^x_{(A^1_t,...,A^h_t)}(x,U_t(x))=0.
\end{equation}
In this way we can reduce our infinite dimensional SPDE to the finite dimensional SDE \refeqn{equation_algorithm5} and to the algebraic (or, more generally, analytic) relations \refeqn{equation_algorithm6}.\\

Let us  now consider Case 2, where the  SPDE is defined by a function $F$, which does not admit characteristics, and by the functions $G_1,...,G_h$ which, as in the previous case, admit characteristics. In this case we can choose a manifold $\mathcal{H}$  defined by the functions
$$u^i_{\sigma}-\partial^{\sigma}(f)(x,a^0),$$
where
$$\partial_{a^0}(f)(x,a^0)=F(x,f(x,a^0),\partial^{\sigma}(f)(x,a^0).$$
In order to obtain the manifold $\mathcal{K}$  as in the previous case,  we require that there exist $h+1$ indices $i_0,...i_h \in \{1,...,n\}$ and $h+1$ multi-indices $\sigma_0,...,\sigma_h \in \mathbb{N}$ such that
\begin{equation}\label{equation_degenerate2}
\left(
\begin{array}{cccc}
D^{\sigma_0}(F^{i_0})(x,f_0(x),\partial^{\tau}(f_0)(x)) & D^{\sigma_0}(G^{i_0}_1)(x,f_0(x),\partial^{\tau}(f_0)(x)) & ... & D^{\sigma_0}(G^{i_0}_h)(x,f_0(x),\partial^{\tau}(f_0)(x))\\
... &... & ... & ...\\
D^{\sigma_h}(F^{i_h})(x,f_0(x),\partial^{\tau}(f_0)(x)) & D^{\sigma_h}(G^{i_h}_1)(x,f_0(x),\partial^{\tau}(f_0)(x)) & ... &D^{\sigma^h}(G^{i_h}_h)(x,f_0(x),\partial^{\tau}(f_0)(x))\end{array} \right),
\end{equation}
 is non singular (here $f_0(x)=f(x,0)$). Therefore, we can define a set of new functions, which we denote by $a^0(x,u,u_{\sigma}),...,a^h(x,u,u_{\sigma})$, such that
$$\bold{\Phi}^*_{(a^1(x,u,u_{\sigma}),...,a^h(x,u,u_{\sigma}))}(u^{i_k}_{\sigma_k})-\partial^{\sigma_k}(f(\bold{\Phi}^x_{a^1(x,u,u_{\sigma}),...,a^h(x,u,u_{\sigma})}(x,u,u_{\sigma_k},a^0(x,u,u_{\sigma}))=0$$
and we can consider the submanifold $\mathcal{K}$ defined as the set of zeros of
$$
\left.\bold{\Phi}^*_{(a^1(x,u,u_{\sigma)},...,a^h(x,u,u_{\sigma}))}(u^i_{\sigma})-\partial^{\sigma}(f^i)(\bold{\Phi}^x_{(a^1(x,u,u_{\sigma)},...,a^h(x,u,u_{\sigma}))}(x,u,u_{\sigma}),a^0(x,u,u_{\sigma}))\right. .
$$

Under these hypotheses for the vector fields $V_{G_i}$  on the manifold $\mathcal{K}$,  an analogue of Theorem \ref{theorem_algorithm} holds.

\begin{theorem}\label{theorem_algorithm2}
In the previous setting $D_i(a^j)|_{\mathcal{K}}=0$ and furthermore
\begin{eqnarray*}
\phi^l_i(a^1,...,a^h)=-\delta^l_i & & \text{when }a^0,a^1,...,a^{i-1}=0\\
\partial_{a^k}(\phi^l_i(a^1,...,a^h))=-\lambda_{k,i}^p \phi^l_p-\sum_{r > k} \psi^r_i \partial_{a^r}(\psi^l_k(a^1,...,a^h))
&&\text{when }a^0,a^1,...,a^{k-1}=0\text{ and }k \geq i,
\end{eqnarray*}
\end{theorem}
\begin{proof}
The proof of this Theorem is completely analogous to the proofs of Lemma \ref{lemma_algorithm1} and Theorem \ref{theorem_algorithm}, exploiting the fact that
$\partial_{b^0}(H(x,f(x,b^0),...))=V_{F}(H)(x,f(x,b^0),...)$ for any function $H \in \mathfrak{F}$ (see Remark \ref{remark_evolution}).
${}\hfill$\end{proof}\\

\noindent Thanks to  Theorem \ref{theorem_algorithm2}, all the machinery developed for Case 1  can be extend to Case 2.
Before  concluding this section, we want to spend few words about the proposed   algorithm  and the non local SPDE  considered in Remark \ref{remark_SPDE1}. In this case equation \refeqn{equation_algorithm5} does not hold, but can  be replaced by
\begin{equation}\label{equation_algorithm8}
dA^i_t=\tilde{\Psi}_{\alpha}(A^1_t,...,A^h_t)\phi^i_{\alpha}(A^1_t,...,A^h_t) \circ dS^{\alpha}_t,
\end{equation}
where $\tilde{\Psi}_{\alpha}(b^1,...,b^h)$ are given by
$$\tilde{\Psi}_{\alpha}(b^1,...,b^h)=\Psi_{\alpha}(K(x,b^1,...,b^h)),$$
and  $K$ is given by relations of the form \refeqn{equation_algorithm7}, that in the particular case $\mathfrak{G}_0=\mathfrak{F}_0$ become
\begin{equation}\label{equation_algorithm9}
\bold{\Phi}^{u^i}_{(b^1,...,b^h)}(x,K(x,b^1,...,b^h))-f(\bold{\Phi}^x_{(b^1,...,b^h)}(x,K(x,b^1,...,b^h))=0.
\end{equation}
We remark that equations \refeqn{equation_algorithm5} and  \refeqn{equation_algorithm6} are not decoupled in the non-local case, and so for solving \refeqn{equation_algorithm9} it becomes essential to write   equations \refeqn{equation_algorithm8}.\\
Anyway, there is one case in which it is possible to get explicitly equation \refeqn{equation_algorithm8} without  solving equation \refeqn{equation_algorithm9}: suppose that $\Psi_{\alpha}(u(x))$ are of the form
$$\Psi_{\alpha}(u(x))=R_{\alpha}(\partial^{\sigma^{\alpha}_1}(u)(h^{\alpha}_1),...,\partial^{\sigma^{\alpha}_r}(u)(k^{\alpha}_r)),$$
for some smooth functions $R_{\alpha}$. If we introduce the new variables $H^{\alpha}_i=\partial^{\sigma^{\alpha}_i}(U_t)(k^{\alpha}_i)$ we can exploit equations \refeqn{equation_algorithm6} in order to prove that $H^{\alpha}_i$ solve the following $A^i_t$ dependent SDE
\begin{equation}\label{equation_algorithm10}
dH^{\alpha}_{j,t}= \mathfrak{H}^{\alpha}_{j,\beta}(H_t,A_t) \circ dS^{\beta}_t
\end{equation}
Indeed, supposing that $H^{\alpha}_{j,t}=U(h^{\alpha}_r)$, we have
\begin{equation}\label{equation_algorithm11}
\begin{array}{c}
\left.\partial_{u}(\bold{\Phi}^{u^i}_{(A^1_t,...,A^h_t)}(x,u)-f(\bold{\Phi}^x_{(A^1_t,...,A^h_t)}(x,u)) )\right|_{u=H^{\alpha}_{j,t}} \circ
dH^{\alpha}_{j,t}+\\
+\left.\partial_{b^i}(\bold{\Phi}^{u^i}_{(b^1,...,b^h)}(x,H^{\alpha}_t)-f(\bold{\Phi}^x_{(b^1,...,b^h)}(x,H^{\alpha}_t))\right|_{b^1=A^1_t,...,b^h=A^h_t}
\circ dA^i_t=0.
\end{array}
\end{equation}
Using equation \refeqn{equation_algorithm11} and equation \refeqn{equation_algorithm8} we obtain  SDE \refeqn{equation_algorithm10}. In this way, even if we are not able to solve  \refeqn{equation_algorithm9}, we can anyway write  explicitly a finite dimensional SDE (given by  equations \refeqn{equation_algorithm8} and \refeqn{equation_algorithm10}), which provides the solution to the initial SPDE.

\section{Examples}\label{section_examplesbis}

\subsection{The proportional volatility HJM model}\label{subsection_HJM}

In this section we consider the problem of finding finite dimensional solutions to the SPDE which naturally arises in the Heath, Jarrow and Morton (HJM)
model to describe the evolution of the interest rate (see \cite{Heath1992}). In this setting, the problem of finding finite dimensional solutions
is called \emph{consistency problem} (see \cite{Bjork1999,Filipovic2001}).
The studies on this topic, and in particular  the works of Filipovic, Tappe and Tiechmann \cite{Filipovic1,Filipovic3,Filipovic2,Tappe2016},
gave us great inspiration for developing the theoretical framework of this paper.
In this section we use our method to provide a closed formula for the solutions to a particular case of HJM model.
Although this SPDE has already been studied, to the best of our knowledge, this is the first time that  an explicit closed formula for its solution is provided.  \\
We consider the following SPDE
\begin{equation}\label{equation_HJM1}
dU_t(x)=\left(\partial_x(U_t)(x)+\Psi(U_t)^2U_t(x)\left(\int_0^x{U_t(y)dy}\right)\right)dt+\Psi(U_t)U_t(x)dW_t,
\end{equation}
where $W_t$ is a Brownian motion and $\Psi: H_w \rightarrow \mathbb{R}$ is a smooth functional defined in  a suitable Hilbert space
$\cinf_0(\mathbb{R}_+) \subset H_w$. Equation \refeqn{equation_HJM1} is closely related to the HJM model. Indeed, if
$$P(t,T)=\exp\left(-\int_t^T{f(s,T)ds} \right)=\exp\left(-\int_t^T{U_{s}(T-s)ds}\right)$$
is the random function which describes the price of a bound at time $t$ with maturity time $T \geq t$, in the HJM framework the evolution of $f$
(called the \emph{forward curve} or the \emph{forward rate}) is described by
$$df(t,T)=f(0,T)+\int_0^t{\alpha(s,T)ds}+\sum_{\beta=1}^k\int_0^t{\sigma_{\beta}(s,T)dW^{\beta}},$$
where
$$\alpha(t,T)=\sum_{\beta=1}^r \sigma_{\beta}(t,T)\int_0^t{\sigma_{\beta}(s,T)ds}$$
and $\sigma_{\beta}(s,T)$ are stochastic predictable processes with respect to $s$. The function $U_t(x)=f(t,t+x)$ is the Musiela parametrization of the forward curve and solves an SPDE of the form \refeqn{equation_HJM1}. In
particular we have equation \refeqn{equation_HJM1} when we choose the volatility of the forward curve $\sigma(t,T)$ proportional to the forward
curve itself
$$\sigma(t,T)=\Psi(f(t,t+x))f(t,T),$$
where $\Psi$ is a functional of the previous form. The proportional HJM model was considered for the first time by Morton  in the
case $\Psi=\Psi_0 \in \mathbb{R}$. In particular, in  \cite{Morton1989},  he proved a result  implying that equation \refeqn{equation_HJM1} has explosion time almost surely
finite (it is possible to choose $\Psi$ non constant such that equation \refeqn{equation_HJM1} has solution for any time $t >0$). In this
subsection we provide an explicit solution formula for equation \refeqn{equation_HJM1}. Although the method used is equivalent to the one proposed
in \cite{Filipovic2} (thus the methods of \cite{Filipovic2} provides the same solution formula) this one seems to be the first time where an  explicit solution formula
is given.\\
In order to explicitly compute the solution to equation  \refeqn{equation_HJM1} we consider the functional space $H_w$ given by
$$H_w=\left\{h\text{ absolutely continuous}\left|\int_0^{+\infty}{(h'(x))^2w(x)dx},\lim_{x \rightarrow +\infty}h(x)=0 \right.\right\}.$$
In \cite{Filipovic2001} it is proved that \refeqn{equation_HJM1} admits a (local in time) unique solution in the Hilbert space $H_w$ when $w$ is
an increasing $C^1$ function such that  $\int_0^{\infty}{w(x)^{-1/3}dx}<+\infty$. \\
The first step to apply our methods to equation \refeqn{equation_HJM1} is  transforming this non-local equation into a local one introducing
a new variable $v$ such that the process $V$ associated with $v$ is
$$V_t(x)=\int_0^x{U_t(y)dy}.$$
With this variable, equation \refeqn{equation_HJM1} becomes
$$dV_t(x)=\left(\partial_x(V_t)(x)-\partial_x(V_t)(0)+\frac{\Psi(\partial_x(V_t))^2(V_t(x))^2}{2}\right)dt+\Psi(\partial_x(V_t))V_t(x) dW_t.$$
If we transform the previous equation into a Stratonovich type equation of the form
\begin{equation}\label{equation_HJM2}
dV_t(x)=\left(\partial_x(V_t)(x)-\partial_x(V_t)(0)+\frac{\Psi(\partial_x(V_t))^2(V_t(x))^2}{2}-\frac{\tilde{\Psi}(\partial_x(V_t))V_t(x)}{2}\right)dt+\Psi(\partial_x(V_t))V_t(x)
\circ dW_t,
\end{equation}
where
$$\tilde{\Psi}(f(x))=\Psi(f(x))^2+\Psi(f(x))\cdot \partial_a(\Psi(e^af(x)))|_{a=0},$$
we can apply the case 1 of the theory proposed in Section \ref{section_general_algorithm} with
\begin{eqnarray*}
G_1&=&v_x\\
G_2&=&1\\
G_3&=&v\\
G_4&=&v^2.
\end{eqnarray*}
It is simple to see that the following commutation relations hold
$$[G_1,G_i]=0, \ \ \ \ [G_2,G_3]=G_2, \ \ \ \ [G_2,G_4]=2G_3, \ \ \ \ [G_3,G_4]=G_4.$$
The characteristic vector fields of $G_1,...,G_4$ are
\begin{eqnarray*}
\bar{V}_{G_1}&=&-\partial_x\\
\bar{V}_{G_2}&=&\partial_v\\
\bar{V}_{G_3}&=&v\partial_v+...+v_{(n)}\partial_{v_{(n)}}+...\\
\bar{V}_{G_4}&=&v^2\partial_v+...+D^n_x(v^2)\partial_{v_{n}}+...
\end{eqnarray*}
They generate the following flows
\begin{eqnarray*}
\Phi^{1,*}_a(x)&=&x-a\\
\Phi^{1,*}_a(v)&=&v\\
\Phi^{2,*}_b(v)&=&v+b\\
\Phi^{3,*}_c(v)&=&e^c v\\
\Phi^{4,*}_d(v)&=&\frac{v}{1-d v}\\
\Phi^{i,*}_k(x)&=&x,
\end{eqnarray*}
for $i=2,3,4$. We take as manifold $\mathcal{H}$ the one dimensional manifold defined by
$$v-f(x)=0,$$
 and by all its differential consequences, where $f$ is smooth, $f(0)=0$, and $f' \in H_w$.\\
The submanifold $\mathcal{K}$, constructed as in Section \ref{section_general_algorithm}, is given by the union of the zeros of
$$\bold{\Phi}_{(a,b,c,d)}(v-f(x))=\frac{e^c(v+b)}{1-de^c(v+b)}-f(x-a),$$
and all its differential consequences. From the particular form of $\mathcal{K}$ we can explicitly compute  the associated finite dimensional function
$$K(x,a,b,c,d)=\frac{e^{-c} f(x-a)}{1+d f(x-a)}-b.$$
Since the operators $\mathfrak{G}_i:v(x) \longmapsto
\partial_x(G_i(v(x)))$ for $i=2,3,4$ are locally Lipschitz in $H_w$ (see \cite{Filipovic2001}), $\partial_x(K) \in H_w$ and the norm
$\|\partial_x(K)\|_{H_w}$ is bounded for $(a,b,c,d)$ in a suitable neighborhood of the origin. For this
reason and for Remark \ref{remark_strong_weak} the finite dimensional solution obtained from $K$ is the unique solution in $H_w$ to equation
\refeqn{equation_HJM1} with initial condition $U_0(x)=f'(x)$.\\
If on $\mathcal{K}$ we choose the coordinate system $(x,a,b,c,d)$ as in Section \ref{section_general_algorithm}, and using Theorem \ref{theorem_algorithm} we obtain
\begin{eqnarray*}
V_{G_1}&=&-\partial_a\\
V_{G_2}&=&-\partial_b\\
V_{G_3}&=&-\partial_c+b\partial_b\\
V_{G_4}&=&-e^{-c}\partial_d+2b\partial_c-b^2\partial_b.
\end{eqnarray*}
With this coordinate system the equation for $A_t,B_t,C_t,D_t$ are
\begin{equation}\label{equation_HJM3}
\begin{array}{rcl}
dA_t&=&-dt\\
dB_t&=&\left(e^{-C_t}\frac{\partial_x(f)(-A_t)}{(1+D_t f(-A_t))^2}-\frac {\Psi_0^2}{2} B_t^2 -\frac {\tilde {\Psi}_0}{2} B_t \right)dt+ \Psi_0 B_t \circ dW_t\\
dC_t&=&\left( \Psi_0 B_t+ \frac {\tilde{\Psi}_0}{2}\right)dt-\Psi_0 \circ dW_t\\
dD_t&=&-\frac {\Psi_0}{2}e^{-C_t}dt
\end{array}
\end{equation}
and the solution to \eqref{equation_HJM1} is given by
\begin{equation}\label{equation_HJM4}
U_t(x)=\partial_x(V_t)(x)=\frac{e^{-C_t}\partial_x(f)(x-A_t)}{(1+D_tf(x-A_t))^2}.
\end{equation} It is evident form the explicit solution \refeqn{equation_HJM4} and from equations \refeqn{equation_HJM3} that the solution
$U_t(x)$ has explosion time almost surely finite as proved by Morton.

\subsection{The stochastic Hunter-Saxton equation}\label{subsection_HS}

In \cite{Holm2016} Holm and Tyranowski propose the following stochastic version of the Camassa-Holm (CH) equation
\begin{equation}\label{equation_Camassa_Holm}
\begin{array}{rcl}
dM_t(x)&=&(-\partial_x(U_t(x)M_t(x))-\partial_x(U_t(x))M_t(x))dt-\sum_{\beta=1}^r(\partial_x(\xi_{\beta}(x)M_t(x))+\\
&&+\partial_x(\xi_{\beta}(x))M_t(x))\circ
dW^{\beta}_t,\\
M_t(x)&=&U_t(x)-\alpha^2\partial_{xx}(U_t(x)).
\end{array}
\end{equation}
This equation is motivated by the study of stochastic perturbations of variational dynamical equation of hydrodynamic type (see
\cite{Arnaudon2012,Cruzeiro2016, Holm2015}). In particular, in \cite{Holm2015} Holm proposes a general method for constructing stochastic
perturbation which preserves some geometrical and physical properties of the considered hydrodynamic PDE. Applying this general principle to the CH
equation in one space dimension we obtain equation \refeqn{equation_Camassa_Holm}. Furthermore, in \cite{Holm2016}
Holm and Tyranowski find that this kind of stochastic perturbation of CH equation preserves the soliton solution, i.e. it is possible to find an
infinite set of finite dimensional solutions to equation \refeqn{equation_Camassa_Holm} which are exactly the stochastic counterpart of the
finite dimensional families of soliton solutions to CH equation. \\
In the following we study this phenomenon in more detail exploiting the methods proposed in the previous section.
Since  equation \refeqn{equation_Camassa_Holm} cannot be directly treated in our framework, being a strongly non local equation, and
it is not possible to transform equation \refeqn{equation_Camassa_Holm} into a local one using the methods proposed
 in Section \ref{subsection_HJM}, we consider a new equation,  related with  \refeqn{equation_Camassa_Holm}, admitting only  finite dimensional solutions.
In particular, equation \refeqn{equation_Camassa_Holm}, in the limit $\alpha >> 1$, can be reduced to the following  stochastic version of Hunter-Saxton equation
\begin{equation}\label{equation_Hunter1}
\begin{array}{rcl}
d\partial_{xx}(U_t(x))&=&(-\partial_x(U_t(x)\partial_{xx}(U_t(x)))-\partial_x(U_t(x))\partial_{xx}(U_t(x)))dt+\\
&&-\sum_{\beta=1}^r(\partial_x(\xi_{\beta}(x)\partial_{xx}(U_t(x)))+\partial_x(\xi_{\beta}(x))\partial_{xx}(U_t(x)))\circ dW^{\beta}_t.
\end{array}
\end{equation}
Choosing a suitable set of possible solutions and the function $\xi_{\beta}(x)$, we can reduce equation \refeqn{equation_Hunter1} to a
weakly local SPDE of the form \refeqn{equation_nonlocal}. In particular,  we consider  $\xi_{\beta}(x)=K_{\beta}+H_{\beta}x$, where $K_{\beta},H_{\beta}$ are suitable constants, and we suppose that the semimartingale $U_t(x)$ depending on the parameter $x$ solution to the equation \refeqn{equation_Hunter1} satisfies
\begin{equation}\label{equation_Hunter2}
\int_{-\infty}^{+\infty}{|x\partial_x(U_t(x))|dx},\int_{-\infty}^{+\infty}{|x\partial_{xx}(U_t(x))|dx},\int_{-\infty}^{+\infty}{|x\partial_{xxx}(U_t(x))|dx} < +\infty. \end{equation} Furthermore we suppose that there are constants  $-\infty \leq a_1<...<a_k \leq +\infty$, for some $k \in \mathbb{N}$, and some constants $C_1,...,C_k \in \mathbb{R}$
such that
$$\sum_{i=1}^k C_i U_t(a_i)=const.$$
Under these conditions, integrating equation \refeqn{equation_Hunter1} first for $\int_\infty^x$ and then for $\int_{a_i}^x$ we obtain that equation \refeqn{equation_Hunter1} is equivalent to the following set of relations
\begin{eqnarray}
&&\begin{array}{rcl}
dU_t(x)&=&(-U_t(x)\partial_x(U_t(x))+\frac{1}{2}V_t(x)+\Psi_0(U_t(x)))dt+\\
&&-\sum_{\beta=1}^r{((K_{\beta}+H_{\beta}x)\partial_x(U_t(x))+\Psi_{\beta}(U_t(x)))}
\circ dW^{\beta}_t\end{array}\label{equation_Saxton1}\\
&&\begin{array}{rcl}
dV_t(x)&=&(-U_t(x)\partial_x(V_t(x))+\Xi_0(U_t(x)))dt+\\
&&-\sum_{\beta=1}^r{((K_{\beta}+H_{\beta}x)\partial_x(V_t(x))+H_{\beta}V_t(x)+\Xi_{\beta}(U_t(x)))}\circ
dW^{\beta}_t\end{array}\label{equation_Saxton2}\\
&&V_t(x)=\sum_{i=1}^k C_i \int_{a_i}^x{(\partial_y(U_t(y)))^2dy}\label{equation_Saxton3}\\
&&\sum_{i=1}^k C_i dU_t(a_i)=0,\label{equation_Saxton4}
\end{eqnarray}
where
\begin{eqnarray*}
\Psi_0(f(x))&=&\sum_{i=1}^k C_i f(a_i)\partial_x(f)(a_i)\\
\Psi_{\beta}(f(x))&=&-\sum_{i=1}^k C_i (K_{\beta}+H_{\beta}a_i)\partial_x(f)(a_i)\\
\Xi_0(f(x))&=&\sum_{i=1}^k C_i(\partial_x(f)(a_i))^2f(a_i)\\
\Xi_{\beta}(f(x))&=&-\sum_{i=1}^k
C_i
(K_{\beta}+H_{\beta}a_i)(\partial_x(f)(a_i))^2
\end{eqnarray*}
It is easy to prove that equation \refeqn{equation_Saxton1} and equation \refeqn{equation_Saxton2} preserve (for $V_t(x),U_t(x)$ smooth in
space) the relation \refeqn{equation_Saxton3}, i.e. if  \refeqn{equation_Saxton3} is satisfied for $t=0$ and $V_t,U_t$ are solutions with
respect to the definition of Remark \ref{remark_SPDE1} to the SPDEs \refeqn{equation_Saxton1} and \refeqn{equation_Saxton2}, then relation
\refeqn{equation_Saxton3} is satisfied for any $t >0$. Furthermore, the three equations \refeqn{equation_Saxton1}, \refeqn{equation_Saxton2} and
\refeqn{equation_Saxton3} imply  \refeqn{equation_Saxton4} if $U_t,V_t$ are smooth. Thus, equation \refeqn{equation_Hunter1}, with
solutions satisfying \refeqn{equation_Saxton4} and with behaviour at infinity given by \refeqn{equation_Hunter2}, is equivalent to the two
dimensional SPDE \refeqn{equation_Saxton1} and \refeqn{equation_Saxton2} with initial conditions $U_0(x),V_0(x)$ satisfying
\refeqn{equation_Saxton3}.\\

In the following we show how it is possible to construct infinite smooth finite dimensional solutions to equation \refeqn{equation_Saxton1} and \refeqn{equation_Saxton2}. If we
consider the following smooth functions in $\mathfrak{F}^2$
\begin{eqnarray*}
G_1&=&\left(\begin{array}{c}
xu_x\\
xv_x+v\end{array} \right),\\
G_2&=&\left(\begin{array}{c} uu_x-\frac{1}{2}v \\
uv_x\end{array} \right),\\
G_3&=&\left(\begin{array}{c}
u_x\\
v_x\end{array} \right),\\
G_4&=&\left(\begin{array}{c} 1 \\
0\end{array} \right),\\
G_5&=&\left(\begin{array}{c} 0 \\
1\end{array} \right),
\end{eqnarray*}
it is easy to see that the functions $G_i$ admit strong characteristics and  generate a Lie algebra with commutation relations given by
\begin{center}
\begin{tabular}{r|c|c|c|c|c|}
$[\cdot,\cdot]$ &$G_1$ & $G_2$ & $G_3$ & $G_4$ & $G_5$\\ \hline
$G_1$& $0$ & $G_2$ & $G_3$ & $0$ & $-G_5$ \\ \hline
$G_2$& $-G_2$ & $0$ & $0$ & $-G_3$ &  $\frac{1}{2}G_4$\\ \hline
$G_3$& $-G_3$ & $0$ & $0$ & $0$ & $0$\\ \hline
$G_4$& $0$ & $G_3$ & $0$ & $0$ & $0$  \\ \hline
$G_5$& $G_5$ & $-\frac{1}{2}G_4$ & $0$ & $0$ & $0$\\ \hline
\end{tabular}
\end{center}
Furthermore $G_1,...,G_5$ admit characteristic vector fields which are
\begin{eqnarray*}
\bar{V}_{G_1}&=&V_{G_1}-xD_x=-x\partial_x+v\partial_v+u_x\partial_{u_x}+2v_x\partial_{v_x}+...+
nu_{(n)}\partial_{u_{(n)}}+(n+1)v_{(n)}\partial_{v_{(n)}}+...\\
\bar{V}_{G_2}&=&V_{G_2}-uD_x=-u\partial_x-\frac{v}{2}\partial_u+...+\left(D^n_x(uu_x)-\frac{v_{(n)}}{2}-uu_{(n+1)}\right)\partial_{u_{(n)}}+\\
&&+\left(D^n_x(uv_x)-uv_{(n+1)}\right)\partial_{v_{(n)}}+...\\
\bar{V}_{G_3}&=&V_{G_3}-D_x=-\partial_x\\
\bar{V}_{G_4}&=&V_{G_4}=\partial_u\\
\bar{V}_{G_5}&=&V_{G_5}=\partial_v.
\end{eqnarray*}
Using the characteristic flows of the vector fields $\bar{V}_{G_i}$ it is possible to apply the results of previous sections. In order to simplify the treatment of this example we suppose that $k=2$, $a_1=-\infty,a_2=+\infty$.
In this case, since we are looking for solutions satisfying $\lim_{x \rightarrow \pm \infty}\partial_x(U_t)(x) = \lim_{x \rightarrow \pm \infty}x \partial_x(U_t)(x) = 0$, we have that $\Psi_0=\Psi_{\beta}=\Xi_0=\Xi_{\beta}=0$. This means that equations \refeqn{equation_Saxton1} and \refeqn{equation_Saxton2} are  local SPDEs. Furthermore, in this case, using the theory of stochastic characteristics  it is possible to prove that, for any smooth initial conditions, there exists a unique (local in time) solution. This means that the smooth finite dimensional solutions which we found with our algorithm are the unique solutions to equation \refeqn{equation_Hunter1} such that $C_1 U_t(-\infty)+ C_2 U(+\infty)=const$ and equation \refeqn{equation_Hunter2} hold. \\

Since $\Psi_0=\Psi_{\beta}=\Xi_0=\Xi_{\beta}=0$, we can consider only the functions $G_1,G_2$ and $G_3$. The most general one dimensional submanifold $\mathcal{H}$ in $\jinf(\mathbb{R},\mathbb{R}^2)$ is defined by the equations
\begin{eqnarray*}
&g_1:=u-f(x)=0&\\
&g_2:=v-g(x)=0&
\end{eqnarray*}
together with all their differential consequences $D^n_x(g_i)=0$. In order to have a manifold $\mathcal{H}$ representing a possible initial condition for our problem, we require condition \refeqn{equation_Hunter2} for $U_0(x)=f(x)$ and that
$$g(x)=C_1 \int_{-\infty}^x{(f'(y))^2dy}+C_2 \int_{+\infty}^x{(f'(y))^2dy}.$$
The first step of our algorithm is to consider the flows of the characteristic vector fields $\bar{V}_{G_i}$ ($i=1,2,3$) given by
\begin{eqnarray*}
\Phi^{1,*}_a(x)&=&e^{-a} x\\
\Phi^{1,*}_a(u)&=&u\\
\Phi^{1,*}_a(v)&=&e^a v\\
\Phi^{2,*}_b(x)&=&x-b u+\frac{b^2}{4}v\\
\Phi^{2,*}_b(u)&=&u-\frac{b}{2}v\\
\Phi^{2,*}_b(v)&=&v\\
\Phi^{3,*}_c(x)&=&x-c\\
\Phi^{3,*}_c(u)&=&u\\
\Phi^{3,*}_c(v)&=&v.
\end{eqnarray*}
Therefore, the manifold  $\mathcal{K}$ is defined by the union on $(a,b,c) \in \mathbb{R}^3$ of the zeros of
\begin{eqnarray*}
\tilde{g}_1(a,b,c,x,u,v)=\bold{\Phi}^*_{(a,b,c)}(g_1)=u-\frac{be^{a}}{2}v-f\left(e^{-a}x-bu+\frac{b^2}{4}e^{a}v-c\right)\\
\tilde{g}_2(a,b,c,x,u,v)=\bold{\Phi}^*_{(a,b,c)}(g_2)=e^{a}v-g\left(e^{-a}x-bu+\frac{b^2}{4}e^{a}v-c\right),
\end{eqnarray*}
and all their differential consequences with respect to $x$.
If, on the manifold $\mathcal{K}$, we  use the coordinate system $(x,a,b,c)$, exploting Theorem \ref{theorem_algorithm}, the three vector fields read
\begin{eqnarray*}
V_{G_1}&=&-\partial_a\\
V_{G_2}&=&-e^{-a}\partial_b\\
V_{G_3}&=&-e^{-a}\partial_c.
\end{eqnarray*}
So the process $A_t,B_t,C_t$  generating the solutions to the considered SPDEs are
\begin{eqnarray*}
dA_t&=&-\sum_{\beta=1}^rH_{\beta}dW^{\beta}_t\\
dB_t&=&e^{-A_t}dt\\
dC_t&=&\sum_{\beta=1}^r K_{\beta}e^{-A_t} \circ dW^{\beta}_t.
\end{eqnarray*}
Since the system for $A,B,C$ is triangular, it can be solved explicitly using only iterated Riemann and It\^o integrals e.g. when one fixes the initial conditions $A_0=B_0=C_0=0$. In this way we obtain the solution to the SPDEs \refeqn{equation_Saxton1} and \refeqn{equation_Saxton2} with initial condition
$U_0(x)=f(x)$ and $V_0(x)=g(x)$.\\
The solutions $U_t(x)$ and $V_t(x)$ to the initial SPDE can be obtained solving the following system of non-linear equations
\begin{eqnarray}
U_t(x)-\frac{B_te^{A_t}}{2}V_t(x)-f\left(e^{-A_t}x-B_tU_t(x)+\frac{B_t^2e^{A_t}}{4}V_t(x)-C_t\right)&=&0\label{equation_Saxton_final1}\\
e^{A_t}V_t(x)-g\left(e^{-A_t}x-B_tU_t(x)+\frac{B_t^2e^{A_t}}{4}V_t(x)-C_t\right)&=&0.\label{equation_Saxton_final2}
\end{eqnarray}
Since $f,g$ are bounded and with bounded derivatives, the system \refeqn{equation_Saxton_final1} and \refeqn{equation_Saxton_final2} admits a unique
solution whenever $A_t,B_t,C_t$ are in a suitable neighborhood of the origin. The fact that the system \refeqn{equation_Saxton_final1} and
\refeqn{equation_Saxton_final2} admits a solution only if $A,B,C$ are suitably bounded is related to the fact that the solutions to the deterministic
Hunter-Saxton equation develop singularity in the first derivative in finite time (see, e.g., \cite{Hunter1991}). This property is conserved by the stochastic perturbation
considered here and it is consistent with analogous results on stochastic CH equation obtained in \cite{Crisan2017}.\\

\begin{remark}
If we choose the functions $\xi_{\beta}(x)$ in equation \refeqn{equation_Hunter1}
different from $K_{\beta}+H_{\beta}x$, not only we are no longer able to reduce equation \refeqn{equation_Hunter1}
 to a local one, but the generic
solution to \refeqn{equation_Hunter1} is not finite dimensional. This does not mean that equation \refeqn{equation_Hunter1} has only
infinite dimensional solutions. Indeed it is possible to verify, using the procedure proposed in \cite{Holm2016}, that equation
\refeqn{equation_Hunter1} has infinite many families of (weak) finite dimensional solutions of the form
$$U_t(x)=\sum_{i=1}^k P^i_t |x^i-Q^i_t|,$$
where $\sum_{i=1}^k P^i_t=0$ and the process $(P^i,Q^i)$ solves a finite dimensional SDE. The class of SPDEs possessing a large set of
families of finite dimensional solutions of increasing dimension does not reduce to equations of the form \refeqn{equation_Camassa_Holm} or
\refeqn{equation_Hunter1}: indeed we provide another example in Section
\ref{subsection_filtering}. In our opinion the class of SPDEs which, despite not having all finite dimensional solutions, possess many families of
finite dimensional solutions deserves more attention and a further detailed investigation.
\end{remark}

\subsection{A stochastic filtering model}\label{subsection_filtering}

In this section we consider an equation inspired by stochastic filtering. In particular, given two  stochastic
processes  $X_t \in \mathcal{X}$ and $Y_t \in \mathcal{Y}$, where $\mathcal{X}$ and $\mathcal{Y}$ are two metric spaces (for example
 $\mathcal{X}=\mathbb{R}^k$ and $\mathcal{Y}=\mathbb{R}^h$), stochastic filtering theory faces the problem of describing the conditional probability
 $\mathbb{P}^X(\cdot |Y)$ of the process $X$ given the process of observation $Y$. Although this one  is in general an infinite dimensional problem,
 there are situations where it is possible to partially describe the probability $\mathbb{P}^X(\cdot |Y)$ using only a finite dimensional process
 $B_t$ on a finite dimensional manifold $B$. When the filtering problem can be reduced to a finite dimensional process we speak of
  finite dimensional filters. Examples of such filters are the Kalman filter, the  Benes filter and the related ones (see \cite{Albeverio1990,Bain,Hazewinkel1989}).\\
In many cases it is possible to reduce the problem of finding and studying finite dimensional filters to the problem of calculating finite
dimensional solutions to particular SPDEs. Indeed, if $\mathcal{X}=M \subset \mathbb{R}^m$ and the process $X_t$, conditioned with
respect to the process $Y_t$, solves a Markovian Brownian-motion-driven SDE, it is possible to describe the filtering problem using a second
order linear SPDEs. There are different ways to obtain this description (in the following we use two of them). The most common  method is to
study a function $\rho_t(x)$ related to the conditional density $p_t(x)$ of the random variable $X_t$ on $M$ conditioned with respect to
$Y_{[0,t]}$. In particular, it is possible to prove that $\rho_t(x)$ solves a second order linear SPDEs called Zakai equation. A
 finite dimensional filter is a filtering problem whose Zakai equation admits (some or all, depending on the definition) finite dimensional solutions.\\
This is the first problem where the research of finite dimensional solution to an SPDEs was studied in detail. Indeed, the theory proposed in
the previous sections has been deeply influenced by the research in this field, and  in particular by the works of Cohen de Lara
\cite{DeLara1bis,DeLara2}.
With our algorithm it is possible to calculate all the solutions to Zakai equation associated with the finite dimensional filters appearing in the previous literature.\\
Instead of applying our algorithm to some already well studied finite dimensional filter, in this section we propose a new filtering problem for which
we are able to calculate some finite dimensional solutions. \\
In particular, we consider the following SPDE
\begin{equation}\label{equation_filtering1}
\begin{array}{rcl}
dU_t(x)&=&\left( \frac{\sigma^2}{2}x\partial^2_x(U_t)(x)+\beta\partial_x(U_t)(x)+\alpha x\partial_x(U_t)(x)+\gamma x U_t(x)+\delta
U_t(x)\right)dt\\
&&+\partial_x(U)_t(x) \circ dS^1_t+ x\partial_x(U_t)(x) \circ dS^2_t,
\end{array}
\end{equation}
where $\sigma,\beta,\alpha,\gamma,\delta$ are some constants and $S^1,S^2$ are the semimartingales driving the equation (below some restrictions on these constants and semimartingales will be discussed), which is related to several  problems of stochastic filtering.\\
For example, if $S^1_t=0$ and $\beta=-\tilde{\beta} <0,\alpha=-\tilde{\alpha} \leq 0, \gamma=0, \delta=-\tilde{\alpha}$  and
$\sigma^2>\tilde{\beta}$  equation
\refeqn{equation_filtering1} is the Zakai equation giving the density of the conditioned probability of the following
filtering problem
\begin{equation}\label{equation_filtering2}
dX_t=(\beta+\alpha X_t)dt+\sigma \sqrt{X_t}dW_t+ X_t d\tilde{S}^2_t,
\end{equation}
with observation given by
\begin{equation}\label{equation_filtering3}
d\tilde{Y}_t=d\tilde{S}^2_t,
\end{equation}
where $\tilde{S}^2_t$ is any semimartingale independent from $W_t$ and $S^2=\tilde{S}^2_t-\frac{1}{2}[\tilde{S},\tilde{S}]_t$.
Equation \refeqn{equation_filtering2} can be considered as a general affine  continuous process perturbed by a noise linearly
dependent on the process itself. It is well known that one dimensional continuous markovian affine processes admit closed form for their
probability densities. Unfortunately the perturbation \refeqn{equation_filtering2} does not admit closed form solution even in the simplest
case where $\tilde{S}^2_t$ is a Brownian motion. \\
In this case the interesting solutions to the SPDE \refeqn{equation_filtering1} should satisfy $U_t(0)=0$, $U_t(x) \geq 0$. These two conditions
guarantee that, if $\int_0^{+\infty}{U_0(x)dx}=1$, then $\int_0^{+\infty}{U_t(x)dx}=1$. Therefore, using the techniques of \cite{Manita2016}, we
can prove that any  solution (smooth in space) to  equation \refeqn{equation_filtering1} is also a solution to the filtering problem
\refeqn{equation_filtering2} and \refeqn{equation_filtering3}. In this case, it is simple to prove that
$$\tilde{F}=\frac{\sigma^2}{2}xu_{xx}+\beta u_x, \ \ G_1=xu_x, \ \ G_2=u  $$
form a three dimensional Lie algebra. For this reason, whenever we know a solution to the equation
$$\partial_a(f(x,a))=\tilde{F}(f(x,a))$$
with $f(0,a)=0$, $f(x,0) \geq 0$ and $\int_0^{+\infty}{f(x,0)dx}=1$, we can apply our technique to equation \refeqn{equation_filtering1}.\\
In particular we consider the two dimensional manifold with boundary $\mathcal{H}$
$$u-f(x,a)=0,$$
for $a\geq 0$. Since the characteristic vector fields
\begin{eqnarray*}
\bar{V}_{G_1}&=&-x\partial_x+u_x\partial_{u_x}+...+nu_{(n)}\partial_{u_{(n)}}+...\\
\bar{V}_{G_2}&=&u\partial_u+u_x\partial_{u_x}+...+u_{(n)}\partial_{u_{(n)}}+...
\end{eqnarray*}
have characteristic flows
\begin{eqnarray*}
\Phi^{1,*}_b(x)&=&e^{-b}x\\
\Phi^{1,*}_b(u)&=&u\\
\Phi^{2,*}_c(x)&=&x\\
\Phi^{2,*}_c(u)&=&e^{c}u,
\end{eqnarray*}
the manifold $\mathcal{K}$ is defined by the union on $(a,b,c) \in \mathbb{R}_+ \times \mathbb{R}^2$ solution to
$$e^{c}u-f(e^{-b}x,a)=0,$$
and all its differential consequences. Using the coordinate system $(x,a,b,c)$ on $\mathcal{K}$, by Theorem \ref{theorem_algorithm2}, we have
\begin{eqnarray*}
V_{\tilde{F}}&=&e^{-b}\partial_a\\
V_{G_1}&=&-\partial_b\\
V_{G_2}&=&-\partial_c.
\end{eqnarray*}
Therefore, the solutions to  SPDE \refeqn{equation_filtering1} can be found solving the following triangular system
\begin{eqnarray*}
dA_t&=&e^{-B_t}dt\\
dB_t&=&-(\alpha dt+ dS^{2}_t)\\
dC_t&=&-\delta dt,
\end{eqnarray*}
and the finite dimensional solution to the Zakai equation is given by
$$U_t(x)=e^{-C_t}f(e^{-B_t}x,A_t).$$\\
Another interesting problem described by equation \refeqn{equation_filtering1} for $S^1_t$ not equal to zero is the
filtering problem
\begin{equation}\label{equation_filtering4}
dX_t=(\beta+ \alpha X_t)dt+\sigma \sqrt{X_t}dW_t+d\tilde{S}^1_t+X_t d\tilde{S}^2_t
\end{equation}
with observations
\begin{equation}\label{equation_filtering5}
\begin{array}{c}
dY^1_t=d\tilde{S}^1_t\\
dY^2_t=d\tilde{S}^2_t,
\end{array}
\end{equation}
where we suppose that the $\mathbb{R}^2$ semimartingale $(\tilde{S}^1,\tilde{S}^2)$ is independent from $W_t$ (instead we do not request that
$\tilde{S}^1$ and $\tilde{S}^2$ are independent). Although for a general noise $\tilde{S}^1$ the solution $X_t$ to equation \refeqn{equation_filtering4}
does not remain positive for all the times $t$, it is possible to provide  sufficient conditions in order to ensure  that this is the case. Suppose that
$\tilde{S}^1$ is almost surely of bounded variation. This means that there are an increasing predictable process $\hat{S}^{1,i}$ and a
decreasing predictable process $\hat{S}^{1,d}$ such that $\tilde{S}^1=\hat{S}^{1,i}+\hat{S}^{1,d}$. If $\hat{S}^{1,d}$ is absolutely continuous
and
$$\frac{d\hat{S}^{1,d}_t}{dt}+\beta>\frac{\sigma^2}{2},$$
for $t >0$ and for any solution $X_t$ to equation \refeqn{equation_filtering4} such that $X_0 >0$ almost surely, we have that   $X_t>0$ almost surely for any $t >0$.
The Zakai equation of filtering problem \refeqn{equation_filtering4} and \refeqn{equation_filtering5} has exactly the form \refeqn{equation_filtering1}.
Unfortunately, for a deep reason that will be clarified below, we cannot deal directly with the Zakai equation of the filtering problem
\refeqn{equation_filtering4} and \refeqn{equation_filtering5} and we have to consider  another SPDE related with this filtering problem.\\
Given  a bounded function $g \in C^2((0,+\infty))$, let us  consider the process dependent on the space parameter $x \in (0,+\infty)$
\begin{equation}\label{equation_filtering6}
V_{t}(x)=\mathbb{E}\left[g(X_T)e^{\int_t^T{(\gamma X_s+\delta)ds}}\left| \{X_t=x\}\vee\mathcal{G}_{t,T} \right.\right],
\end{equation}
where
$$\mathcal{G}_{t,T}=\sigma\{\tilde{S}^1_T-\tilde{S}^1_s,\tilde{S}^2_T-\tilde{S}^2_s|s \in [t,T]\}.$$
The process $V_t$ is adapted with respect to the inverse filtration $\mathcal{G}_{t,T}$ with $t \in (0,T]$. If
$\tilde{S}^2_{T}-\tilde{S}^2_{T-t}$ is a semimartingale with respect to the filtration $\mathcal{G}_{T,t}$ (an example of such processes is given by
Brownian motions or solutions to Markovian Brownian motion driven SDEs),  we can generalize Theorem 2.1 of \cite{Pardoux1979} (see also
\cite{Bally2001,Pardoux1994}) proving that $U_t(x)=V_{T-t}(x)$ solves equation \refeqn{equation_filtering1} with
$S^1_t=\tilde{S}^1_T-\tilde{S}^1_{T-t}$ and $S^2_t=\tilde{S}^2_T-\tilde{S}^2_{T-t}-\frac{1}{2}[\tilde{S}^2_T-\tilde{S}^2_{T-t},\tilde{S}^2_T-\tilde{S}^2_{T-t}]_t$. If we can explicitly find solutions to  equation \refeqn{equation_filtering1}, we have a closed formula
for the conditional expected value \refeqn{equation_filtering6}, extending in this way the closed formula of some expected values of Markovian
continuous affine one dimensional processes.\\
It is important to note that any bounded smooth solution to equation \refeqn{equation_filtering1} is a solution to the problem
\refeqn{equation_filtering6} since such kinds of solutions are unique (this fact can be proven using the coordinate change
$\tilde{x}=e^{-S^2_t}x$ and then using some standard reasoning based on the maximum principle for parabolic PDEs see, e.g.
Theorem 4.1 and Theorem 4.3 of \cite{Friedman1975}). For all these reasons we are interested in finding solutions to equation
\refeqn{equation_filtering1} when $\beta>0$ and $\gamma<0$.\\
We remark that, if $S^2$ is not identically zero, we do not have a finite dimensional Lie algebra. Indeed, in this case, if we put
\begin{eqnarray*}
F&=&xu_{xx}\\
G_3&=&u_x\\
G_4&=&xu,
\end{eqnarray*}
we have
\begin{equation}\label{equation_commutator}
\underset{n\text{ times }}{[F,[F,[...[F,u_x]...]]]}=n! u_{(n+1)}
\end{equation}
and so $F,G_1,G_2,G_3,G_4$ cannot form a finite dimensional Lie algebra on all the space $\jinf(\mathbb{R}_+,\mathbb{R})$. This means that
the  solution $U_t(x)$ to equation \refeqn{equation_filtering1} with a general initial condition $U_t(x)=f(x)$ is
not finite dimensional. This is why we choose to consider the problem \refeqn{equation_filtering6} instead of the Zakai equation related
to the filtering problem \refeqn{equation_filtering4} and \refeqn{equation_filtering5}. Indeed, a smooth solution to the Zakai equation on
$(0,+\infty)$ should satisfy $U_t(0)=0$ in order to be the conditional probability
density
of the filtering problem \refeqn{equation_filtering4} and \refeqn{equation_filtering5}. Unfortunately we are not able to construct solutions to  equations of the form
\refeqn{equation_filtering1} satisfying this property if $S^1 \not =0$. This is due to the fact that $F,G_1,...,G_4$ do not form a finite dimensional
Lie algebra. Conversely a sufficient condition ensuring that a smooth solution $U_t(x)$ to equation \refeqn{equation_filtering1} represents the
integral \refeqn{equation_filtering6} is that $U_t(x)$ is bounded in $(0,+\infty)$. We are able to construct bounded finite dimensional solutions
to  equation \refeqn{equation_filtering1}, so giving (for suitable functions $g$) the explicit expression of the conditional expectation
\refeqn{equation_filtering6}. \\
In order to construct families of finite dimensional solutions to  equation \refeqn{equation_filtering1} we exploit the particular form of the
commutators \refeqn{equation_commutator}. Indeed let $\mathcal{K}$ be the finite dimensional submanifold of $\jinf(\mathbb{R}_+,\mathbb{R})$
defined by
\begin{equation}\label{equation_filtering7}
h=u_{(n)}+\sum_{k=0}^{n-1}\mu^{k}u_{(k)}=0,
\end{equation} and all its differential consequences with respect to $x$ considering $\mu^k$ as
constants. Using that
\begin{equation}\label{equation_filtering8}
\begin{array}{ccl}
[F,u_{(k)}]&=&k u_{(k+1)}\\
\left[ G_1,u_{(k)} \right]&=& k u_{(k)}\\
\left[ G_2,u_{(k)} \right]&=&0\\
\left[ G_3,u_{(k)} \right]&=&0\\
\left[ G_4,u_{(k)} \right]&=&k u_{(k-1)},
\end{array}
\end{equation}
we are able to prove that $V_{F},V_{G_1},...,V_{G_4} \in T\mathcal{K}$ on $\mathcal{K}$. It is important to note that the previous relation does not hold
 if we consider the submanifold $\tilde{\mathcal{K}}$ defined by equation \refeqn{equation_filtering7} with all its differential
consequences where $\mu^k$ are fixed constant and not variable constants (with respect to $x$). This situation is similar to the
previous section where $\mathcal{K}$ is defined by $\bold{\Phi}^*_{\alpha}(h^i)$ and all its differential consequences, where $h^i=0$ defines the
submanifold $\mathcal{H}$. \\
If we choose on the manifold $\mathcal{K}$ the coordinate system given by $(x,\mu^0,...,\mu^{n-1},u,....,u_{(n-1)})$, it is possible to prove that
$V_F(\mu^k),V_{G_i}(\mu^k)$ depend only on $\mu^0,...,\mu^k$. Furthermore, using  relation \refeqn{equation_filtering8}, it is possible to compute
$V_F(\mu^k),V_{G_i}(\mu^k)$. In order to illustrate  the explicit  calculations, we consider the submanifold $\mathcal{K}$  defined
by
$$h=u_{xx}+\lambda u_x+\mu u=0$$
and we calculate $V_F$. We have that
\begin{eqnarray*}
V_F(h)|_{\mathcal{K}}&=&\left(V_{F}(u_{xx})+\lambda V_F(u_x) + \mu V_F(u)+ V_F(\lambda) u_x+ V_F(\mu) u\right)|_{\mathcal{K}}\\
&=&\left([F,u_{xx}]+\lambda[F,u_x]+\mu [F, u]+V_F(\lambda) u_x+ V_F(\mu) u\right)|_{\mathcal{K}}\\
&=&\left(2u_{xxx}+\lambda u_{xx}+V_F(\lambda) u_x+ V_F(\mu) u\right)|_{\mathcal{K}}\\
&=&\left(2(-\lambda u_{xx}-\mu u_x)+\lambda (-\lambda u_x-\mu u)+V_F(\lambda) u_x+ V_F(\mu) u\right)|_{\mathcal{K}}\\
&=&\left((V_F(\lambda)-2\mu+\lambda^2)u_x+(V_F(\mu)+\lambda \mu)u \right)|_{\mathcal{K}}.
\end{eqnarray*}
Since $u_x,u$ can take any values on $\mathcal{K}$ we can find the expression for $V_F(\lambda),V_F(\mu)$ on $\mathcal{K}$. Using similar
methods we have
\begin{eqnarray*}
V_{F}&=&(-\lambda^2+2\mu)\partial_{\lambda}-\mu\lambda\partial_{\mu}-x(\lambda u_x+\mu u)\partial_u-((\lambda u_x+\mu u)-x(\lambda^2-\mu)u_x-x\lambda\mu u)\partial_{u_x}\\
V_{G_1}&=&\lambda\partial_{\lambda}+2\mu\partial_{\mu}+xu_x\partial_{u}+(u_x-x(\lambda u_x+\mu u))\partial_{u_x}\\
V_{G_2}&=&u\partial_u+u_x\partial_{u_x}\\
V_{G_3}&=&u_x\partial_{u}-(\lambda u_x+\mu u)\partial_{u_x}\\
V_{G_4}&=&-2\partial_{\lambda}-\lambda\partial_{\mu}+xu\partial_u+(xu_x+u)\partial_{u_x}.
\end{eqnarray*}
The SDE for $L_t,M_t,U_t(0),U_{x,t}(0)$ becomes
\begin{eqnarray*}
dL_t&=&\left(\frac{\sigma^2}{2}(-L_t^2+2M_t)+\alpha L_t-2\gamma \right)dt+L_t \circ dS^2_t\\
dM_t&=&\left(-\frac{\sigma^2}{2}L_tM_t+2\alpha M_t-\gamma L_t\right)dt+2M_t \circ dS^2_t\\
\left(\begin{array}{c} dU_t(0)\\
dU_{x,t}(0) \end{array}\right)&=&\left(\begin{array}{cc} \delta & \beta\\
-M_t\left(\frac{\sigma^2}{2}+\beta\right)+\gamma & -L_t\left(\frac{\sigma^2}{2}+\beta \right)+\alpha +\delta \end{array} \right) \cdot \left(\begin{array}{c} U_t(0)\\
U_{x,t}(0) \end{array}\right) dt+\\
&& + \left(\begin{array}{cc} 0 & 1\\
-M_t & -L_t\end{array} \right) \cdot \left(\begin{array}{c} U_t(0)\\
U_{x,t}(0) \end{array}\right) \circ dS^1_t+\left(\begin{array}{cc} 0 & 0\\
1 & 0\end{array} \right) \cdot \left(\begin{array}{c} U_t(0)\\
U_{x,t}(0) \end{array}\right) \circ dS^2_t
\end{eqnarray*}
The solution to  SPDE \refeqn{equation_filtering1} can be obtained solving the system
\begin{eqnarray*}
\partial_x(U_{t})(x)&=&U_{x,t}(x)\\
\partial_x(U_{x,t}(x))&=&-L_t U_{x,t}(x)- M_t U_t(x).
\end{eqnarray*}
Thus we have two possibilities. If $\lambda_0^2-4\mu_0>0$ then
$$U_t(x)=A_te^{C_t x}+B_t e^{D_t x},$$
where
\begin{eqnarray*}
C_t&=&\frac{-L_t+\sqrt{L_t^2-4M_t}}{2}\\
D_t&=&\frac{-L_t-\sqrt{L_t^2-4M_t}}{2} \\
A_t&=&\frac{D_tU_t(0)-U_{x,t}(0)}{D_t-C_t}\\
B_t&=&\frac{-C_tU_t(0)+U_{x,t}(0)}{D_t-C_t}.
\end{eqnarray*}
If $\lambda_0^2-4\mu_0<0$ we have
$$U_t(x)=e^{R_tx}(A_t \cos(O_tx)+B_t \sin(O_t x)),$$
where
\begin{eqnarray*}
R_t&=&\frac{-L_t}{2}\\
O_t&=&\frac{\sqrt{4M_t-L_t^2}}{2}\\
A_t&=&U_t(0)\\
B_t&=&\frac{-R_tU_t(0)+U_{x,t}(0)}{O_t}.
\end{eqnarray*}

\section*{Acknowledgements}

The author would like to thank Prof. Paola Morando for her  useful comments, suggestions and corrections of the first draft of the paper. This work was supported by Gruppo Nazionale Fisica Matematica (GNFM-INdAM) through the grant: \virgolette{Progetto Giovani, Symmetries and reduction for differential equations: from the deterministic to the stochastic case}.

\bibliographystyle{plain}
\bibliography{finite_dimensional(6)}

\end{document}